\documentclass[a4,leqno]{amsart}
\usepackage{amsmath}
\usepackage{amssymb}
\usepackage{amsthm}

\usepackage{nccmath}

\usepackage[usenames]{color}
\usepackage{eepic,epic}

\newtheorem{thm}{Theorem}[section]
\newtheorem*{theorem*}{Theorem}

\newtheorem{lem}[thm]{Lemma}
\newtheorem{prop}[thm]{Proposition}
\theoremstyle{definition}
\newtheorem{defn}[thm]{Definition}
\theoremstyle{remark}
\newtheorem{rem}[thm]{Remark}
\numberwithin{equation}{section}

\begin{document}

\title[Optimal bounds for Oscillating convolution operators on the Heisenberg group]
{$L^2$ and $H^p$ boundedness of strongly singular operators and oscillating operators on Heisenberg groups}
\author{Woocheol Choi}
\address{Department of Mathematical Sciences, Seoul National University, Seoul 151-747, Korea}
\email{chwc1987@snu.ac.kr}
\thanks{The author was supported by Global Ph.D Fellowship of the government of South Korea.}


\maketitle
\begin{abstract} In this paper we establish sharp $L^2$ and $H^p$ boundedness results for strongly singular operators and oscillating operators on Heisenberg groups. 
%Strongly singular convolution operator on the Heisenberg group $\mathbb{H}^{n}_a$ was introduced in \cite{lyall}. The $L^2$ boundedness property for the operator was partially obtained in \cite{lyall} and extended by Laghi and Lyall \cite{laghi}. However, optimal condition was not still achieved fully. In this paper, we obtain the sharp results for all cases. In addition we extend the boundedness to Hardy spaces. On hand, we define the convolution operator with oscillating kernels on the Heisenberg group and study its boundedness, which was studied on Euclidean space (see \cite{Sj2, Sj3}). 
%
\end{abstract}

\section{The reviewer's comments}
Reviwer's comment: As I said, this is a very well written article. One minor typo that I noticed was two incidences where the Laghi-Lyall[10] was simply referred to as Laghi[10](page 3 line 6 and page 8 line 16), this should be changed.
\

$\rightarrow$ We changed the two incidences by the blue colored text ( page 3 line 12 and page 8 line 24).

\section{Introduction}
The setting of this paper is the Heisenberg group $\mathbb{H}^n_a$, $a \in \mathbb{R}^{*}$, realized as  $\mathbb{R}^{2n+1}$ equipped with the group law,
\begin{eqnarray*}
(x,t) \cdot (y,s) = (x+ y, s+t -2 a x^{T} J y), \qquad J =
\begin{pmatrix} 0 & I_n \\ - I_n & 0
\end{pmatrix}.
\end{eqnarray*}
 This group is equipped with the following anisotropic dilations,
\begin{eqnarray*}
\lambda \cdot (x,t) = (\lambda x , \lambda^2 t ), \qquad \lambda > 0.
\end{eqnarray*}
For $K \in \mathcal{D}'(\mathbb{H}_a^n)$ we denote by $T_K$ the convolution operator defined by $K$, i.e, 
\begin{eqnarray*}
T_K f (x,t)  : = K * f (x,t) = \int_{\mathbb{H}^n_a} K\left( (x,t) \cdot (y,s)^{-1} \right) f(y,s) dy dx, \qquad f \in C_0^{\infty} (\mathbb{H}_a^{n}).
\end{eqnarray*}
We say that the operator $T_K$ is bounded on $L^{p}(\mathbb{H}^n)$ if there exist a $C > 0$ such that
\begin{eqnarray*}
\| T_K f \|_p \leq C \|f \|_p, \quad \textrm{for all } ~ f \in C^{\infty}_{0} ( \mathbb{H}^n_a). 
\end{eqnarray*}
A natural quasi-norm on the Heisenberg group is given by 
\[
\rho (x,t) = (|x|^4 + t^2 )^{1/4}, \qquad (x,t)\in \mathbb{H}_a^n.
\]This quasi-norm satisfies $\rho (\lambda \cdot(x,t)) = \lambda \rho (x,t)$.
For this quasi-norm, we define the strongly singular kernels,
\[
 K_{\alpha, \beta} (x,t) ={\rho(x,t)^{- (2n+2+\alpha)}}e^{i\rho(x,t)^{-\beta}} \chi ( \rho(x,t)), \quad \alpha>0, \quad \beta>0,
\]
where $\chi$ is a smooth bump function in a small neighborhood of the origin. This operator was introduced by Lyall~\cite{lyall} who showed that $T_{K_{\alpha,\beta}}$ is bounded when $\alpha \leq n \beta$. This result was obtained by using the Fourier transform on the Heisenberg group in combination with involved estimates on oscillatory integrals. Subsequently, Laghi-Lyall~\cite{laghi} obtained sharp results in the special case $a^2 < C_{\beta}$ (where $C_{\beta}$ is given by~(\ref{eq:Intro.Cbeta})) by using a version for the Heisenberg group of the $L^2$-boundedness theorem for non-degenerate oscillatory integral operators of H\"ormander~\cite{ho}. 
In this paper, we shall consider the cases $a^2 \geq C_{\beta}$ and obtain sharp conditions using the  theory for oscillatory integral operators with degenerate phases. 

Strongly singular convolution operators were originally considered on $\mathbb{R}^n$. Such operators correspond to suitable oscillating multipliers. They were first studied, by Fourier transform techniques, in the Euclidean setting with $\rho(x) = |x|$ by Hirschman  \cite{hi}, Wainger~\cite{Wainger}, Fefferman~\cite{Fe}, and Fefferman-Stein \cite{Fe2}. 
\

In addition the convolution operator with kernel of the form $ \frac{1}{|x|^{n-\alpha}} e^{i |x|^{\beta}}, \alpha, \beta >0$, has been investigated in the last decades. Such kernels have no singularity near the origin, but they assume relatively small decaying property at infinity. The case $\beta=1$ corresponds to the kernel of Bochner-Riesz means. For $\beta \neq 1$, the $(L^p,L^q)$ estimates and Hardy space estimates has been completely studied by Miyachi~\cite{mi}, Pan-Sampson~\cite{PSam} and  Sj\'olin~\cite{Sj1, Sj2,Sj3}. The difference between the two cases comes from the fact that the phase kernel $|x-y|^{\beta}$ is degenerate only if $\beta =1$.  
In this paper, we also consider the analogous problem on the Heisenberg groups for the following kernels,
\[
L_{\alpha, \beta}(x,t) = {\rho(x,t)^{- (2n+2-\alpha)}}e^{i\rho(x,t)^{\beta}} \chi (\rho(x,t)^{-1}), \quad \beta > 0.
\]
We denote by $T_{L_{\alpha,\beta}}$ the group convolution operators with the kernel $L_{\alpha,\beta}$. 
\

In the first part of this paper, we shall find the optimal ranges of $\alpha$ and $\beta$ where the convolution operators associated with $K_{\alpha,\beta}$ and 
$L_{\alpha, \beta}$ are bounded on $L^2 (\mathbb{H}_a^n)$.
\

Before stating our results, we recall the  previous results of Laghi-Lyall~\cite{laghi} and Lyall \cite{lyall}. Set
\begin{equation}
C_{\beta} = \frac{\beta+2}{2} ( 2\beta + 5 + \sqrt{(2\beta+5)^2 -9}).
\label{eq:Intro.Cbeta}
\end{equation}
Then we have 
\begin{theorem*}[Laghi-Lyall~\cite{laghi}, Lyall \cite{lyall}]\mbox{~}
\begin{enumerate}
\item $T_{K_{\alpha,\beta}}$ is bounded on $L^2  (\mathbb{H}^n_a)$ if $ \alpha \leq n \beta$.% (see \cite{lyall}).
\item  If $ 0 < {a^2} < C_{\beta}$, then $T_{K_{\alpha,\beta}}$ is bounded on $L^2 (\mathbb{H}_a^n)$ if and only if $\alpha \leq (n+ 1/2) \beta$.% (see \cite{laghi}).
\end{enumerate}
\end{theorem*}
We shall prove the sharp $L^2$ boundedness results for $T_{K_{\alpha,\beta}}$ when $a^2 \geq C_{\beta}$. 
\begin{thm}\label{main1}~
\begin{enumerate} 
\item If ${a^2} > C_{\beta}$, then $T_{K_{\alpha,\beta}}$ is bounded on $L^2 (\mathbb{H}^n_a)$ if and only if $\alpha \leq (n+ \frac{1}{3})\beta$. 
\item If ${a^2} = C_{\beta}$, then $T_{K_{\alpha,\beta}}$ is bounded on $L^{2}(\mathbb{H}^n_a)$ if and only if $\alpha \leq (n+\frac{1}{4})\beta$.
\end{enumerate}
\end{thm}

For the operators $T_{L_{\alpha,\beta}}$, we shall obtain the sharp $L^2$ boundedness results except the case $\beta = 1$ and the case $\beta = 2$.
\begin{thm}\label{main2}~ 
\begin{enumerate}
\item If $ 0 < \beta < 1$, then $T_{L_{\alpha,\beta}}$ is bounded on $L^2$ if and only if one of the following condition holds.
\begin{enumerate}
\item[(i)]   ${a^2} < C_{\beta}$ and $\alpha \leq (n+\frac{1}{2}) \beta$,
\item[(ii)] ${a^2} = C_{\beta}$ and $\alpha \leq ( n+ \frac{1}{4})\beta$,
\item[(iii)] ${a^2} > C_{\beta}$ and $\alpha \leq (n+ \frac{1}{3}) \beta$.
\end{enumerate}
\item If $ 1 < \beta < 2$, then $T_{L_{\alpha,\beta}}$ is bounded on $L^2$ if and only if $\alpha \leq (n+\frac{1}{3}) \beta$.
\item If $ 2 < \beta  $, then $T_{L_{\alpha,\beta}}$ is bounded on $L^2$ if and only if  $\alpha \leq (n+ \frac{1}{2}) \beta$.
\end{enumerate}
\end{thm}
 In \cite{laghi} Laghi-Lyall reduced the boundedness problem for operators on the Heisenberg group to that for the local operators and used a version of H\"omander's $L^2$-boundedness theorem on the Heisenberg group. However, as we shall show, we may view the operators on the Heisenberg group as operators on Euclidean space $\mathbb{R}^{2n+1}$. This will enable us to use the oscillatory integral estimates of Greenleaf-Seeger~\cite{GR} and Pan-Sogge~\cite{PS} on Euclidean space. 
 %In fact we shall be concerned about the uniformity for the reduced local operators since we have some unbounded factors in the group law. 
 %Nevertheless, this does not cause any problem since the determinant of matrix for the coordinates changes is equal to 1 (see \eqref{det1} in Section 2).

For the cases $\beta = 1$ or $\beta = 2$, folds with degree $>3$ appear in the reduced local oscillatory integral operators for some values of $a$. The sharp estimates for degerate oscillatory integral estimates has been achieved for degree less or equal to $3$ (see Greenleaf-Seeger~\cite{GR} and Pan-Sogge~\cite{PS}). We hope to address the remaining problem in the future.

\

 For $p > 1$,  $L^p$ boundedness can be obtained by interpolation between the $L^2$ boundedness estimates and some $L^1$ boundedness estimates for dyadic-piece operator. We refer to {\textcolor{blue}{Laghi-Lyall}~\cite[Theorem 5]{laghi}} for the case $a^2 <  C_{\beta}$ except the endpoint. Using the interpolation technique, we shall get the $L^p$ boundedness in the case $a^2 \geq C_{\beta}$.
\begin{thm}\label{LP1}~
\begin{enumerate}
\item If $a^2 > C_{\beta}$, then $T_{K_{\alpha,\beta}}$ is bounded on $L^p (\mathbb{H}_a^{n})$ if $\alpha -(n+ \frac{1}{3})\beta < 2 \beta (n+ \frac{1}{3}) \left|\frac{1}{p} - \frac{1}{2}\right|$.
\item If $a^2 = C_{\beta}$, then $T_{K_{\alpha,\beta}}$ is bounded on $L^p(\mathbb{H}_a^{n})$ if $\alpha -(n+ \frac{1}{4})\beta < 2 \beta ( n+ \frac{1}{4}) \left|\frac{1}{p} -\frac{1}{2}\right|$.
\end{enumerate}
\end{thm}
\begin{thm}\label{LP2}~
\begin{enumerate}
\item If $0< \beta<1$, then $T_{L_{\alpha,\beta}}$ is bounded on $L^p(\mathbb{H}_a^n)$ if one of the following holds.
\begin{enumerate}
\item[(i)] $a^2 < C_{\beta}$ and $ \alpha - (n+\frac{1}{2})\beta < 2\beta (n+ \frac{1}{2})\left|\frac{1}{p} -\frac{1}{2}\right|$,
\item[(ii)] $a^2 = C_{\beta}$ and $ \alpha - (n+ \frac{1}{4})\beta < 2 \beta (n+\frac{1}{4}) \left|\frac{1}{p} -\frac{1}{2}\right|$,
\item[(iii)] $a^2< C_{\beta}$ and $\alpha - (n+ \frac{1}{3})\beta < 2 \beta (n+ \frac{1}{3})\left|\frac{1}{p} -\frac{1}{2}\right|$.
\end{enumerate}
\item If $1 < \beta <2$, then $T_{L_{\alpha,\beta}}$ is bounded on $L^p(\mathbb{H}_a^n)$ if $\alpha - 2(n+\frac{1}{3})\beta < 2 \beta (n+ \frac{1}{3}) \left|\frac{1}{p} -\frac{1}{2}\right|$.
\item If $2 < \beta$, then $T_{L_{\alpha,\beta}}$ is bounded on $L^p(\mathbb{H}_a^n)$ if $\alpha -2 (n+ \frac{1}{2}) \beta < 2 \beta (n+\frac{1}{2}) \left|\frac{1}{p} -\frac{1}{2}\right|$.
\end{enumerate}
\end{thm}

The second part of this paper is devoted to prove the boundedness on Hardy spaces $H^p$ ($p \leq 1$) of the operators $T_{K_{\alpha,\beta}}$ and $T_{L_{\alpha,\beta}}$. On Euclidean space the boundedness on Hardy spaces was proved up to the endpoint cases by Sj\'olin~\cite{Sj1, Sj3}. In this case, the operator can be thought as a multiplier operator $T f = {(m \widehat{f})}^{\vee}$ and we have the relation $c_p \sum_{j=1}^{n} \| R_j f \|_{L^p} \leq \| f\|_{H^p} \leq C_p \sum_{j=1}^{n} \| R_j f \|_{L^p}$ and we see that derivatives of the symbol $\frac{\xi_j}{|\xi|} m(\xi)$ of the multiplier $R_j m(D)$ are pointwisely bounded by the derivatives of the symbol $m(\xi)$. These things make it possible to calculate the $H^p$ norm accurately to obtain the sharp boundedness result including for the endpoint cases 
(see Miyachi~\cite{mi}). 

The above outline seems difficult to adapt to the Heisenberg group. Instead we shall make use of the molecular decomposition of Hardy spaces. Then we obtain the following result.
\begin{thm}\label{main3} Let $p \in (0,1)$ and let $\alpha$ and $\beta$ be real numbers such that $(\frac{1}{p} -1)(2n+2)\beta + \alpha < 0$. Then 
\begin{enumerate}
\item The operator $T_{K_{\alpha,\beta}}$ is bounded on $H^p$ space.
\item For $\beta \neq 1$, the operator $T_{L_{\alpha,\beta}}$ is bounded on $H^p$ space.
\end{enumerate}
These conditions are optimal except for the endpoint case $(\frac{1}{p} -1 )(2n+2)\beta + \alpha = 0$. 
\end{thm} 
This paper is organized as follows. In Section 2, we reduce the $L^2$ boundeness problem on the Heisenberg group to a local oscillatory integral estimates on Euclidean space. In Section 3, we recall some essential results for the oscillatory integral operators with degenerate phase functions and study geometry of the canonical relation and projection maps associated with the phase functions of the reduced operators. Then, we  will complete the proof of Theorem \ref{main1} and Theorem \ref{main2}. In section 4, we recall some background on hardy spaces on the Heisenberg group and its basic properties. In section 5, we prove Theorem \ref{main3}. In Section 6, we show that the conditions of Theorem \ref{main3} are sharp except the endpoint cases.
\

\section*{Notation}
We will use the notation $\lesssim$ instead of $\leq C$ when the constant $C$ depends only on the fixed parameters such as $a, \alpha, \beta$ and n. In addition, we will use the notation $A \sim B$ when both inequalities $A \lesssim B$ and $A \gtrsim B$ hold.

\section{Dyadic decomposition and Localization}

In this section we reduce our problems to some oscillatory integral estimates problem on Euclidean space $\mathbb{R}^{2n+1}$. This reduction is well-known for  operators on Euclidean space (see Stein \cite{St}). The issue of this reduction on the Heisenberg group is to control the localized operators $\tilde{T}_j^{k,l}$ in \eqref{A} uniformly for $(g_k, g_l)$ such that $\rho(g_k \cdot g_l^{-1}) \leq 2$. Note that the cut-off functions $\eta( \rho((x,t) \cdot g_k^{-1}))~\eta(\rho((y,s)\cdot g_l^{-1}))$ have no uniform bound for their derivatives. Nevertheless we get the uniformity after a value-preserving change of coordinates (see \eqref{det1}).
\

We decompose the kernels $K_{\alpha, \beta}$ and $L_{\alpha, \beta}$ as 
\begin{eqnarray}\label{decom1}
K_{\alpha, \beta} (x,t) &=& \sum_{j=1}^{\infty} K_{\alpha,\beta}^{j}, \quad   K_{\alpha,\beta}^{j} : =  \eta (2^j \rho(x,t)) K_{\alpha, \beta} (x,t),
\end{eqnarray}
and
\begin{eqnarray}\label{decom2}
L_{\alpha, \beta} (x,t) &=& \sum_{j=1}^{\infty} L_{\alpha,\beta}^{j} , \quad L_{\alpha,\beta}^{j} : = \eta (2^{-j} \rho(x,t)) L_{\alpha, \beta} (x,t),
\end{eqnarray}
where $\eta \in C^{\infty}_{0}(\mathbb{R})$ is a bump function supported in $[\frac{1}{2},2]$ such that $\sum_{j=0}^{\infty} \eta(2^j r) =1$ for all $ 0 < r \leq 1$. 
 For notational convenience, we omit the index $\alpha$ and $\beta$ from now on.
\\
Set $T_j f = K_{\alpha, \beta}^{j} * f$ and $S_j f = L_{\alpha,\beta}^{j} * f$. Then we have
\begin{lem}
For each $N \in \mathbb{N}$, there exist constants $C_N >0$ and $c_{\beta} >0$ such that
\begin{eqnarray*}
\| T_j^{*} T_{j'} \|_{L^2 \rightarrow L^2} + \| T_j T_{j'}^{*}\|_{L^2 \rightarrow L^2} \leq C_N 2^{-max\{j, j'\}N}
\\
\| S_j^{*} S_{j'} \|_{L^2 \rightarrow L^2} + \| S_j S_{j'}^{*}\|_{L^2 \rightarrow L^2} \leq C_N 2^{-max\{j, j'\}N}
\end{eqnarray*}
holds for all $j$ and $j'$ satisfying $|j - j'|  \geq c_{\beta}$.
\end{lem}
\begin{proof} The proof follows from the integration parts technique in the typical way, so we omit the details. See Lyall \cite[Lemma 2.4]{lyall} where the proof for $T_j$ is given.
\end{proof}
 By Cotlar-Stein Lemma, we only need to show that there is a constant $C>0$ such that
\begin{eqnarray*}
\| T_j \|_{L^2 \rightarrow L^2} + \| S_j \|_{L^2 \rightarrow L^2} \leq C \quad \forall j \in \mathbb{N}. 
\end{eqnarray*}
We consider the dilated kernels
\begin{equation}\label{tildek}
\begin{split}
\tilde{K}^{j}_{\alpha,\beta} (x,t) =& K_{\alpha,\beta}^{j}(2^{-j} \cdot (x,t)) = \eta(\rho(x,t)) 2^{j(Q+\alpha)} \rho(x,t)^{-Q-\alpha} e^{i2^{j\beta} \rho(x,t)^{-\beta}},
\\
\tilde{L}^{j}_{\alpha,\beta} (x,t) =& L_{\alpha,\beta}^{j}(2^{-j} \cdot (x,t)) = \eta(\rho(x,t)) 2^{-j(Q-\alpha)} \rho(x,t)^{-Q+\alpha} e^{i2^{j\beta} \rho(x,t)^{\beta}}.
\end{split}
\end{equation}
We define $\tilde{T}_j$ and $\tilde{S}_j$ to be the convolution operators with kernels given by $\tilde{K}_{\alpha,\beta}^{j}$ and $\tilde{L}_{\alpha,\beta}^{j}$. Set $f_j (x,t) = f(2^{-j} \cdot (x,t))$. Then $K_{\alpha,\beta}^{j}*f (2^{-j}\cdot (x,t)) = 2^{-jQ} (\tilde{K}_{\alpha,\beta}^{j} * f_j ) (x,t)$, and we have

\begin{eqnarray}\label{Red1}
\begin{split}
\|T_j f \|_{L^2} = \|K^{j}_{\alpha,\beta} * f (x,t) \|_{L^2} = & 2^{-jQ/2} \| K_{\alpha,\beta} * f (2^{-j} \cdot (x,t))\|_{L^2}
\\
\leq & 2^{-jQ/2}\cdot 2^{-jQ} \|\tilde{K}_{\alpha,\beta}^{j} * f_j  (x,t) \|_{L^2}
\\
\leq& 2^{-jQ/2} \cdot 2^{-jQ} \|\tilde{T}_{j} \|_{L^2 \rightarrow L^2} \| f_j \|_{L^2}
\\
\leq& 2^{-jQ} \|\tilde{T}_{j} \|_{L^2 \rightarrow L^2} \|f\|_{L^2}.
\end{split}
\end{eqnarray}
Similarly, we have $\|S_j f \|_{L^2} \leq 2^{jQ} \| \tilde{S}_{j} \|_{L^2 \rightarrow L^2} \| f \|_{L^2}$. It follows  that it is enough to prove that $\| \tilde{T}_j \|_{L^2 \rightarrow L^2} \lesssim 2^{jQ}$ and $\|\tilde{S}_j \|_{L^2 \rightarrow L^2} \lesssim 2^{-jQ}$. 
\

Now, we further modify our operators to some operators defined locally using the fact that the kernels of $\tilde{T}_j$ and $\tilde{S}_j$ are supported in $\{ (x,t) : \rho(x,t) \leq 2 \}$. To do this we find a set of point $ G = \{ g_k : k \in \mathbb{N} \}$ such that $\bigcup_{k\in \mathbb{N}} B(g_k, 2 ) = \mathbb{H}^{n}_a $ and each $B( g_k, 4)$ contains only $d_n$'s other $g_l$ members in $G$.
\

We can split $f = \sum_{k=1}^{\infty} f_k$ with each $f_k$ supported in ${B(g_k ,2 )}$. Define
\begin{eqnarray}\label{A}
\tilde{T}_j^{k,l} f (x,t) = \int \tilde{K}_{\alpha,\beta}^{j} \left( (x,t) \cdot (y,s)^{-1} \right) \cdot \eta \left( \rho\left( (x,t) \cdot g_k^{-1}\right) \right) 
\eta \left( \rho\left( (y,s) \cdot g_l^{-1}\right) \right)  f(y,s) dy ds.
\end{eqnarray}
Then, 
\begin{equation}\label{Red2}
\begin{split}
\| \tilde{T}_j * f \|^{2}_{L^2(\mathbb{H}_a^{n})} \leq& \sum_{k=1}^{\infty} \| \tilde{T}_j * f \|^2_{L^2(B(g_k,2))} 
\\
\leq& \sum_{k=1}^{\infty} \| \tilde{T}_j * \sum_{l=1}^{\infty} f_l \|_{L^2 ( B(g_k,2))}^2
\\
\leq& \sum_{k=1}^{\infty} \| \tilde{T}_j * \sum_{\{ l : \rho(g_l \cdot g_k^{-1}) \leq 2 \}} f_l \|_{L^2((B(g_k,2)))}^{2}
\\
\lesssim& \sum_{k=1}^{\infty} \sum_{l : \rho(g_l \cdot g_k^{-1}) \leq 2} \|\tilde{T}_{j}^{k,l} \|_{L^2 \rightarrow L^2} \| f_l \|_{L^2}^2
\\
\lesssim&  \sup_{\substack{\rho ( g_l \cdot g_k^{-1}) \leq 2}} \|\tilde{T}_j^{k,l} \|_{L^2 \rightarrow L^2}  \|f \|_{L^2}^{2}.
\end{split}
\end{equation} 
We note that
\begin{eqnarray}\label{det1} 
\det \left(D_{x,t} \left( (x,t) \cdot g\right) \right) = 1 ~\textrm{for all }~g \in \mathbb{H}^{n}_a.
\end{eqnarray} 
Then, using the coordinate change $(y,s) \rightarrow ((y,s) \cdot g_k)$ and substituting $(x,t) \rightarrow ((x,t) \cdot g_k)$ in \eqref{A}, we get

\begin{eqnarray}\label{Red3}
\begin{split}
&\tilde{T}_j^{k,l} f \left( (x,t)\cdot g_k\right)
\\
&\quad \quad \quad \quad = \int \tilde{K}^{j}_{\alpha,\beta} \left( (x,t)\cdot (y,s)^{-1}\right) \eta(\rho(x,t))\eta(\rho((y,s)\cdot (g_k\cdot g_l^{-1}))) f((y,s)\cdot g_k) dy ds.
\end{split}
\end{eqnarray}
Notice that $\rho(g_k\cdot g_l^{-1}) \lesssim 1$. Set $\psi \left( (x,t),(y,s)\right)= \eta (\rho(x,t)) \eta(\rho((y,s)\cdot (g_k\cdot g_l^{-1})))$ and write $f$ just for $ f ( ()\cdot g_k)$. Then $\sup_{\rho(g_l \cdot g_k^{-1}) \leq 2} \| \tilde{T}_j^{k,l}\|$  will be achieved if we prove $\| \mathcal{T}_j \|_{L^2 \rightarrow L^2} \lesssim 2^{jQ}$ for 
\begin{eqnarray}\label{Red4}
\mathcal{T}_j f (x,t) = \int \tilde{K}_{\alpha,\beta}^{j} \left( (x,t)\cdot(y,s)^{-1} \right) \psi \left( (x,t) , (y,s) \right) f (y,s) dy ds
\end{eqnarray}
with a compactly supported smooth function $\psi$. Finally we set
\begin{equation}\label{AB}
\begin{split}
A_j (x,t) =& 2^{j\alpha} \mu(x,t) e^{i 2^{j\beta} \rho(x,t)^{-\beta}},
\\
B_j (x,t) =& 2^{j\alpha} \mu (x,t) e^{i 2^{j\beta} \rho(x,t)^{\beta}},
\end{split}
\end{equation}
where $\mu$ is a smooth function supported on the set $\{ (x,t) \in \mathbb{R}^{2n+1} : \frac{1}{10} \leq \rho(x,t) \leq 10 \}.$
We define the operators $L_{A_j}$ and $L_{B_j}$ by
\begin{eqnarray*}
L_{A_j} f (x,t) = \int A_j \left( (x,t) \cdot (y,s)^{-1}\right) \psi\left( (x,t), (y,s)\right) f(y,s) dy ds,
\\
L_{B_j} f(x,t) = \int B_j \left( (x,t) \cdot (y,s)^{-1} \right) \psi\left( (x,t),(y,s)\right) f(y,s) dy ds.
\end{eqnarray*}
We shall deduce Theorem \ref{main1} and Theorem \ref{main2} from the following propositions.
\begin{prop}\label{LA} ~
\begin{enumerate} 
\item If $a^2 > C_{\beta}$, then 
\[ \| L_{A_j}\|_{L^2 \rightarrow L^2} \lesssim 2^{j (\alpha - (n+\frac{1}{3})\beta)}, \qquad \forall j \in \mathbb{N}.\]
\item If $a^2 = C_{\beta}$, then 
\[ \|L_{A_j}\|_{L^2 \rightarrow L^2} \lesssim 2^{j (\alpha - (n+\frac{1}{4})\beta)}, \qquad \forall j \in \mathbb{N}.\]
\end{enumerate}
\end{prop}

\begin{prop}\label{LB}~ 
\begin{enumerate}
\item If $ 0 < \beta < 1, $ then,
\begin{enumerate}
\item[(i)] For $a^2 < C_{\beta}$, 
\[\|L_{B_j}\|_{L^2 \rightarrow L^2} \lesssim 2^{j (\alpha -(n+\frac{1}{2})\beta)} \qquad \forall j \in \mathbb{N}.\]
\item[(ii)] For $a^2 = C_{\beta}$, 
 \[\|L_{B_j}\|_{L^2 \rightarrow L^2} \lesssim 2^{j (\alpha - (n+\frac{1}{4})\beta)} \qquad \forall j \in \mathbb{N}.\]
\item[(iii)] For $a^2 > C_\beta$, 
\[ \| L_{B_j}\|_{L^2 \rightarrow L^2} \lesssim 2^{j(\alpha - (n+\frac{1}{3})\beta)} \qquad \forall j \in \mathbb{N}.\]
\end{enumerate}
\item If $ 1 < \beta < 2,$ then 
\[ \|L_{B_j}\|_{L^2 \rightarrow L^2} \lesssim 2^{j (\alpha - (n+\frac{1}{3})\beta)} \qquad \forall j \in \mathbb{N}.\]
\item If $2 < \beta $, then \[ \| L_{B_j} \|_{L^2 \rightarrow L^2} \lesssim 2^{j (\alpha -(n+\frac{1}{2})\beta)} \qquad \forall j \in \mathbb{N}.\]
\end{enumerate}
\end{prop}
We get the first main result of this paper assuming these propositions:
\begin{proof}[Proof of Theorem \ref{main1} and Theorem \ref{main2}]
From the reductions \eqref{Red1}, \eqref{Red2} and \eqref{Red3}, in order to prove Theorem \ref{main1} it is enough to prove that $\|\mathcal{T}_j \|_{L^2 \rightarrow L^2} \lesssim 2^{jQ}$ for the operators $\mathcal{T}_j$ given in \eqref{Red4}. From \eqref{tildek} and \eqref{AB} we have $\mathcal{T}_j = 2^{jQ} L_{A_j}$ with a suitable function $\mu$, and so $\|\mathcal{T}_j\|_{L^2 \rightarrow L^2} = 2^{jQ} \| L_{A_j}\|_{L^2 \rightarrow L^2}$. Therefore, the estimates of Proposition \ref{LA} yield Theorem \ref{main1}. In the same way, Proposition \ref{LB} establishes Theorem \ref{main2}.
\end{proof}
\begin{proof}[Proof of Theorem \ref{LP1} and Theorem \ref{LP2}]
By the duality argument, it is enough to prove for $p <2$. In addition, we shall prove  only the case (1) of Theorem \ref{LP1}, the other cases will follow from the same argument. Suppose $p<2$ and $a^2 >C_{\beta}$. Since $\| T_j \|_{L^2 \rightarrow L^2} \lesssim \|L_{A_j}\|_{L^2 \rightarrow L^2}$, Proposition \ref{LA} yields 
\begin{eqnarray*}
\| T_j \|_{L^2 \rightarrow L^2} &\lesssim& 2^{j (\alpha -(n+\frac{1}{3})\beta)}
\end{eqnarray*}
On the other hand, Young's inequality gives
\begin{eqnarray*}
\| T_j \|_{L^1 \rightarrow L^1} &\lesssim& 2^{j (\alpha )}.
\end{eqnarray*}
Interpolating above two estimates we get
\begin{eqnarray*}
\| T_j \|_{L^p \rightarrow L^p} \lesssim 2^{j (\alpha 2(\frac{1}{p} -\frac{1}{2})  + (\alpha - (n+\frac{1}{3})\beta) 2(1-\frac{1}{p}))} = 2^{j (\alpha - 2(n+\frac{1}{3})(1-\frac{1}{p}))}.
\end{eqnarray*} 
Thus, we may sum the geometric series if $ \alpha - 2(n+\frac{1}{3})(1-\frac{1}{p})$. This completes the proof.
\end{proof}
In the next section, we shall briefly review on the theory related to the operators $L_{A_j}$ and $L_{B_j}$.  We will make use of geometric properties of the phase function $\rho(x,t)^{\beta}$ to prove Proposition \ref{LA} and Proposition \ref{LB}.
\section{$L^2$ estimates}
We begin with the $L^2 \rightarrow L^2 $ theory for oscillatory integral operators. The operators we are concern with are of the form
\begin{eqnarray*}
T^{\phi}_{\lambda} f (x) = \int_{\mathbb{R}^n} e^{i \lambda \phi(x,y)} a(x,y) f (y) dy,
\end{eqnarray*}
where $\phi \in C^{\infty}(\mathbb{R}^n \times \mathbb{R}^n)$ and $a \in C^{\infty}_c(\mathbb{R}^n \times \mathbb{R}^n)$. Suppose that the phase function $\phi$ satisfies $\det\left( \frac{\partial^2 \phi}{\partial x_i \partial y_j} \right) \neq 0$ on the support of $a$, we say that $\phi$ is non-degenerate. We say that $\phi$ is degenerate if there is some point $(x_0, y_0)$ where $\det \left.\left(\frac{\partial^2}{\partial x_i \partial y_j}\right)\right |_{(x_0,y_0)}$ equals to zero. For non-degenerate phases, we have the fundamental theorem of H\"ormander.
\begin{thm}[H\"ormander \cite{ho}] Suppose that the phase function $\phi$ is non-degenerate. Then we have
\begin{eqnarray*}
\|T^{\phi}_{\lambda} \|_{L^2 \rightarrow L^2} \lesssim \lambda^{-\frac{n}{2}} \quad 
 \forall \lambda \in [1, \infty).
\end{eqnarray*}
\end{thm}

This theorem gives sharp decaying rate of the norm $\|T_{\lambda}^{\phi}\|_{L^2\rightarrow L^2}$ in terms of $\lambda$. However, the phase functions of our operators $L_{A_j}$ and $L_{B_j}$ can become degenerate according to the values of $a$ and $\beta$ (see Lemma \ref{det} and Lemma \ref{factor}). For a degenerate phase function $\phi$, the optimal number $\kappa_{\phi}$ for which the inequality $\|T_{\lambda}\|_{L^2 \rightarrow L^2} \lesssim \lambda^{-\kappa_{\phi}}$ holds would be less than $\frac{n}{2}$. The number $\kappa_{\phi}$'s are  related to the type of fold of the phase $\phi$ (see Definition \ref{def}). For phases whose types of folds are $\leq 3$, the sharp numbers $\kappa_{\phi}$ were obtained by Greenleaf-Seeger \cite{GR} and Pan-Sogge \cite{PS}. We shall use the results. The sharp results for folding types $\leq 3$ in \cite{GR} are the best known results and there are no optimal results for folding types $>3$ except some special cases established by Cuccagna \cite{cu}. 

It is well-known that the decaying property is strongly related to the geometry of the canonical relation,
\begin{eqnarray}\label{relation1}
C_{\phi} = \{ (x, \partial_x \phi(x,y), y, - \partial_y \phi(x,y) )~ ; ~x, y \in \mathbb{R}^n \} \subset T^{*} (\mathbb{R}^n_x) \times T^{*} (\mathbb{R}^n_y).
\end{eqnarray}
%To describe the geometry, we let $\pi : M \rightarrow N$ a smooth map between %$C^{\infty} $manifolds $M$ and $N$ of same dimension. We assume that the corank %of $d\pi$ is 1 and $d(\det(d\pi))$ does not vanish in a neighborhood of the %hypersurface
%\begin{eqnarray*}
%\Sigma(\pi) = \{ p \in M \mid \det(d\pi)|_p = 0 \}.
%\end{eqnarray*}
%Let $V$ be a smooth vector field defined in a neighborhood $U$ of a point $p_0 \in %\Sigma(\pi)$, which satisfies
%\begin{eqnarray*}
%V|_U \neq 0,\qquad V|_{U \cap \Sigma(\pi)} \in \textrm{Ker} ~d\pi.
%\end{eqnarray*}
%Then we define the type of $\pi$ at a point $p_0 \in \Sigma(\pi)$ as the smallest $k %\in \mathbb{N}$ such that $V^{k} (\det d\pi)|_{p_0} \neq 0$. 

\begin{defn}\label{def}
Let $M_1$ and $M_2$ be smooth manifolds of dimension $n$, and let $f : M_1 \rightarrow M_2$ be a smooth map of corank $\leq 1$. Let $S = \{ P \in M_1 : \textrm{rank} (Df)  < n ~\textrm{at}~ P\}$ be the singular set of $f$. Then we say that $f$ has a $k-$type fold at a point $P_0 \in S$ if
\begin{enumerate}
\item $\textrm{rank}(Df)|_{P_0} = n-1$,
\item $\det (Df)$ vanishes of $k$ order in the null direction at $P_0$.
\end{enumerate}
Here, the null direction is the unique direction vector $v$ such that $(D_v f)|_{P_0} =0$.
\end{defn}
\
Now we consider the two projection maps
\begin{eqnarray}\label{relation2}
\pi_L : C_{\Phi} \rightarrow T^{*} ( \mathbb{R}^n_x) \quad \textrm{and} \quad \pi_R : C_{\Phi} \rightarrow T^{*} (\mathbb{R}^n_y).
\end{eqnarray}
\begin{prop}[\cite{GR},\cite{PS}]\label{degenerate}
Suppose that the projection maps $\pi_L$ and $\pi_R$ have 1-type folds (Whitney folds) singularities, then 
\begin{eqnarray*}
\| T_{\lambda} f \|_{L^2(\mathbb{R}^n)} \lesssim \lambda^{-\frac{(n-1)}{2}-\frac{1}{3}} \| f \|_{L^2(\mathbb{R}^n)} \quad \forall \lambda \in [1, \infty).
\end{eqnarray*}
If the projection maps $\pi_L$ and $\pi_R$ have 2-type folds singularities, then 
\begin{eqnarray*}
\| T_{\lambda} f \|_{L^2(\mathbb{R}^n)} \lesssim \lambda^{-\frac{(n-1)}{2}-\frac{1}{4}} \| f \|_{L^2(\mathbb{R}^n)} \quad \forall \lambda \in [1, \infty).
\end{eqnarray*} 
\end{prop}
In order to use Proposition \ref{degenerate},  we shall study the projection maps \eqref{relation2} associated to the phase function of the operators $L_{A_{j}}$ and $L_{B_{j}}$. Recall that $\rho(x,t) = (|x|^4 + t^2)^{1/4}$ and the phase function $\phi$ of the integral operators $L_{A_j}$ and $L_{B_j}$  is 
\begin{eqnarray*}
\phi(x,t,y,s) = \rho^{-\beta} \left( (x,t)\cdot (y,s)^{-1}\right).
\end{eqnarray*}
To write the group law explicitly, we write $x = (x^1, x^2)$ and $y = (y^1, y^2)$ with $x^j, y^j \in \mathbb{R}^n$. Set $\Phi(x,t) = \rho(x,t)^{-\beta}$. Then 
\begin{eqnarray}\label{phi}
\phi(x, t, y, s) = \Phi \left( x^1 - y^1, x^2 - y^2, t - s - 2a (x^1 y^2 -x^2 y^1)\right).
\end{eqnarray}
For notational purpose set $t = x_{2n+1}$ and $s = y_{2n+1}$. To determine whether the phase function $\Phi$ is non-degenerate, we need to calculate the determinant of the matrix,
\begin{eqnarray*}
H=\left( \frac{\partial^2  \phi(x,t,y,s) }{\partial{y_i} \partial{x_j}} \right).
\end{eqnarray*}
The determinant is calculated in \textcolor{blue}{Laghi-Lyall} \cite{laghi}. However we give a somewhat simpler computation by considering the matrix $L$ associated naturally with the matrix $H$ (see below), which will also be useful in Lemma \ref{rank} and the proof of Proposition \ref{LA} and Proposition \ref{LB}.
\

For simplicity, we write $(\textbf{x}, \textbf{t}) = (x,t) \cdot (y,s)^{-1}$. By the Chain Rule, for $1 \leq i, j \leq n$, we have
\begin{eqnarray*}
\frac{\partial}{\partial x_j} \phi(x,t,y,s) = \left[\partial_{j} + 2 a y_{n+j} \partial_{2n+1} \right] \Phi (\textbf{x}, \textbf{t}),
\\
\frac{\partial}{\partial x_{j+n}} \phi(x,t,y,s) = \left[\partial_{{j+n}} - 2 a y_j \partial_{2n+1} \right] \Phi (\textbf{x}, \textbf{t}).
\end{eqnarray*}
Using the Chain Rule once more, we get
\begin{equation}
\begin{split}
\frac{\partial}{\partial y_i} \frac{\partial}{\partial x_j} \phi (x,t,y,s) = & \left[( \partial_{i} + 2 a x_{n+i} \partial_{2n+1} ) (\partial_{j} + 2a y_{n+j} \partial_{2n+1} )\right] \Phi (\textbf{x}, \textbf{t}),
\\
\frac{\partial}{\partial y_{n+i}} \frac{\partial}{\partial x_j}\phi(x,t,y,s) = & \left[(\partial_{{n+i}} - 2 a x_i \partial_{2n+1} ) (\partial_{j} + 2a y_{n+j} \partial_{2n+1} )\right] \Phi (\textbf{x}, \textbf{t})+ \left[2a \delta_{ij} \partial_{2n+1} \right]\Phi(\textbf{x}, \textbf{t}),
\\
\frac{\partial}{\partial{y_i}} \frac{\partial}{{n+j}}\phi(x,t,y,s) = &\left[ (\partial_{i}+ 2a x_{n+i} \partial_{2n+1} )(\partial_{{n+j}} - 2a y_{j} \partial_{2n+1} ) \right]\Phi (\textbf{x}, \textbf{t}) - \left[ 2 a \delta_{ij}\partial_{2n+1} \right]\Phi(\textbf{x}, \textbf{t}),
\\
\frac{\partial}{\partial{y_{n+i}}} \frac{\partial}{\partial x_{n+j}}\phi(x,t,y,s) = & \left[(\partial_{{n+i}} - 2a x_i \partial_{2n+1} ) (\partial_{{n+j}} - 2 a y_j \partial_{2n+1} ) \right]\Phi(\textbf{x}, \textbf{t}).
\end{split}
\end{equation}
Define
\begin{eqnarray*}
A_a (y) = \begin{pmatrix} I & 2a J y \\ 0 & 1 \end{pmatrix}, \qquad J = \begin{pmatrix} 0 & I_n \\ - I_n & 0 \end{pmatrix}. 
\end{eqnarray*}
Then we have 
\begin{equation}\label{35}
\begin{split}
H(x,t,y,s) =& A_a (x) \left( \partial_{i} \partial_{j} \Phi \right) (\textbf{x}, \textbf{t})~A_a (y)^T + 2 a (\partial_{2n+1} \Phi)(\textbf{x}, \textbf{t})  \begin{pmatrix} J & 0 \\ 0 &0\end{pmatrix}
\\
=& A_a (x)  \left[ (\partial_{i} \partial_{j} \Phi ) + 2 a (\partial_{2n+1} \Phi) \begin{pmatrix} J & 0 \\ 0 &0 \end{pmatrix} \right] (\textbf{x}, \textbf{t}) ~A_a (y)^{T},
\end{split}
\end{equation}
where the second equality holds because {\footnotesize$A_a (x)  \begin{pmatrix}J&0\\0&0\end{pmatrix} A_a (y)^{T} = \begin{pmatrix}J&0\\0&0\end{pmatrix}$.} Set
\begin{eqnarray}\label{36}
L(x,t,y,s) = \left[ (\partial_{i} \partial_{j} \Phi ) + 2 a ( \partial_{2n+1} \Phi) \begin{pmatrix} J & 0 \\ 0 &0 \end{pmatrix} \right] (\textbf{x}, \textbf{t}).
\end{eqnarray}
Thus, to study the matrix $H$, it is enough to analyze the matrix $L$. Moreover we have  $\det(A_a (x)) = \det(A_a (y))=1$ and it implies that $ \det (H(x,t,y,s)) = \det ( L(x,t,y,s)).$ Therefore it is enough to calculate the determinant of $L$.

To find \eqref{36} we calculate the Hessian matrix of $\Phi$. For $1 \leq i, j \leq 2n$,

\begin{eqnarray*}
\partial_{j} \Phi (\textbf{x}
,\textbf{t}) & = &- \frac{\beta}{4} (|{\textbf{x}}
|^4 + \textbf{t}^2)^{-\frac{\beta}{4} -1} ( 4 \textbf{x}
_j |\textbf{x}
|^2),
\\
\partial_{2n+1} \Phi(\textbf{x}
,\textbf{t}) & = & - \frac{\beta}{4}(|\textbf{x}
|^4 + \textbf{t}^2)^{-\frac{\beta}{4}-1}(2\textbf{t}),
\end{eqnarray*}
and
\begin{eqnarray*}
\partial_{i} \partial_{j} \Phi(\textbf{x}
,\textbf{t})&=& \beta(|\textbf{x}
|^4 + \textbf{t}^2)^{-\frac{\beta}{4} -2} \left[ (\beta+4) |\textbf{x}
|^4 - 2 (|\textbf{x}
|^4+ \textbf{t}^2) \right] \textbf{x}
_i \textbf{x}
_j - \beta(|\textbf{x}
|^4 + \textbf{t}^2)^{-\frac{\beta}{4}-1} \delta_{ij} |\textbf{x}
|^2,
\\
\partial_{i} \partial_{2n+1} \Phi(\textbf{x}
,\textbf{t})&=& \beta(\beta+4) (|\textbf{x}
|^4 + \textbf{t}^2)^{-\frac{\beta}{4}-2} |\textbf{x}
|^2  \textbf{x}_i \cdot \frac{\textbf{t}}{2},
\\
\partial_{2n+1}^2 \Phi(\textbf{x}
,\textbf{t}) &=& \beta (\beta+4) (|\textbf{x}
|^4 + \textbf{t}^2 )^{-\frac{\beta}{4} -2}  \frac{\textbf{t}}{2} \cdot \frac{\textbf{t}}{2} - \beta (|\textbf{x}
|^4 + \textbf{t}^2)^{-\frac{\beta}{4}-1} \frac{1}{2}.
\end{eqnarray*}
Set $D = (|\textbf{x}
|^2 \textbf{x}
, \frac{\textbf{t}}{2} )^{T}$. Then the above computations show that
\begin{equation}\label{37}
\begin{split}
~&\left[ (\partial_{i} \partial_{j} \Phi) + 2 a (\partial_{2n+1} \Phi) \begin{pmatrix} J & 0 \\ 0&0 \end{pmatrix} \right]({ \textbf{x}
,\textbf{t} })
\\=& \beta(\beta + 4) (|\textbf{x}
|^4 + \textbf{t}^2 )^{-\frac{\beta}{4} -2} D \cdot D^{T} - \beta (|\textbf{x}
|^4 + \textbf{t}^2)^{-\frac{\beta}{4} -1} \begin{pmatrix} |\textbf{x}
|^2 I + a\textbf{t} J + 2 \textbf{x}
 \cdot \textbf{x}
^{T} & 0 \\ 0 & \frac{1}{2} \end{pmatrix}
\\=&-\beta (|\textbf{x}
|^4 + \textbf{t}^2 )^{-\frac{\beta}{4}-1} ( E + R ),
\end{split}
\end{equation}
where we set 

\begin{eqnarray}\label{BKE} B = |\textbf{x}|^2 I + a \textbf{t}J,\quad K = \textbf{x}\cdot \textbf{x}^T,\quad E =\begin{pmatrix} B+2K  & 0 \\ 0 & \frac{1}{2}\end{pmatrix} \quad\textrm{and} \quad R= - \frac{(\beta+4)}{|\textbf{x}
|^4 + \textbf{t}^2} D\cdot D^{T}.
\end{eqnarray}
Then, from \eqref{36} and \eqref{37} we get
\begin{eqnarray}\label{sowe}
L(x,t,y,s) = [- \beta(|\textbf{x}
|^4 + \textbf{t}^2 )^{-\frac{\beta}{4}-1} (E + R)]~(\textbf{x}, \textbf{t}).
\end{eqnarray}
\begin{lem}\label{det} We have
\begin{eqnarray*}
\det H (x,t,y,s) = F( (x,t) \cdot (y,s)^{-1}),
\end{eqnarray*}
where $F(x,t) = c_{a,\beta} (|x|^4 + a^2 t^2)^{m_1} (|x|^4 + t^2 )^{m_2} f (x,t)$  for some $m_1,m_2, c_{a,\beta} \in \mathbb{R}$ and  $f(x,t) = 2(\beta+1)|x|^8 + (3(\beta+2) - 2a^2)|x|^4 t^2 + (\beta+2) a^2 t^4$.
\end{lem}
\begin{proof} 
We write $(\mathbf{x},\mathbf{t}) = (x,t)\cdot (y,s)^{-1}$ again. In view of \eqref{35}, \eqref{36} and \eqref{sowe}, it is enough to show that
\begin{eqnarray*}
\det [-\beta(|\textbf{x}|^4+ \textbf{t}^2)^{-\frac{\beta}{4}-1} (E + R) ] = F (\textbf{x},\textbf{t}).
\end{eqnarray*}
Considering the form of the function $F$ given, we only need to compute $\det (E+R)$. From \eqref{BKE} we have
\begin{eqnarray*}
E + R = \begin{pmatrix} B + 2 K & 0 \\ 0 & \frac{1}{2} \end{pmatrix} - \frac{(\beta+4)}{|\textbf{x}|^4 + \textbf{t}^2} D \cdot D^{T}.
\end{eqnarray*}
\
For notational convenience, we shall  use lower-case letters $f_1,\dots,f_m$ to denote the  rows of a given $m \times m$ matrix $F$. 
Notice that $D  D^{T} $ is of rank $1$ and we have the following equality
\begin{eqnarray}\label{convention}
\det ( P + Q) = \det(P) + \sum_{j=1}^{m} \det ( p_1^T , \dots, p_{j-1}^T, q_j^T, p_{j+1}^T,\dots ,p_{m}^T ), 
\end{eqnarray}
for any $m \times m$ matrices $P$ and $Q$ with rank $Q =1$. Recall that $B = |\textbf{x}|^2 I + a\textbf{t} J$ and $K=\textbf{x}\cdot \textbf{x}^T$, then direct calculations show that  
\begin{eqnarray}\label{detB}
\det (B) = (|\textbf{x}|^4 + a^2 \textbf{t}^2)^n
\end{eqnarray}
and 
\begin{equation}\label{detK}
\begin{split}
&\sum_{j=1}^{n} \textbf{x}_j \det \begin{pmatrix}  b_1^T, \cdots,  b_{j-1}^T , k_j^T,  b_{j+1}^T, \cdots, b_{2n}^T \end{pmatrix}  + \sum_{j=1}^{n} \textbf{x}_{j+n} \det \begin{pmatrix} b_1^T, \cdots, b_{j+n-1}^T,  k_{j+n}^T, b_{j+n+1}^T, \cdots, b_{2n}^T\end{pmatrix}  
\\
=& \sum_{j=1}^{n} \textbf{x}_j ( |\textbf{x}|^2 \textbf{x}_j + \textbf{x}_{n+j} a \textbf{t}) (|\textbf{x}|^4 + a^2 \textbf{t}^2)^{n-1} + \sum_{j=1}^{n} \textbf{x}_{j+n} (|\textbf{x}|^2 \textbf{x}_{j+n} - \textbf{x}_j a \textbf{t} ) (|\textbf{x}|^4 + a^2 \textbf{t}^2)^{n-1}
\\
=& (|\textbf{x}|^4 + a^2 \textbf{t}^2)^{n-1} |\textbf{x}|^4.
\end{split}
\end{equation}
Thus, from \eqref{convention}, \eqref{detB} and \eqref{detK}, we get
\begin{equation}\label{det2}
\begin{split}
\det (B+2K) =& (|\textbf{x}|^4 + a^2 \textbf{t}^2)^{n} + 2 |\textbf{x}|^4 (|\textbf{x}|^4 +a^2 \textbf{t}^2)^{n-1} 
\\
=& (|\textbf{x}|^4+a^2 \textbf{t}^2)^{n-1}(3|\textbf{x}|^4 + a^2 \textbf{t}^2).
\end{split}
\end{equation}
Using  \eqref{convention} once again, we obtain
\begin{eqnarray*}
\det(E+R) &=& \det(E) + \frac{1}{2} \sum_{j=1}^{2n} \det \begin{pmatrix} e_1 \\ \vdots \\ e_{j-1} \\ r_j \\ e_{j+1} \\ \vdots \\ e_{2n} \end{pmatrix} 
+ \det \begin{pmatrix} e_1 \\ \vdots \\ e_{2n} \\ r_{2n+1} \end{pmatrix}
\\
&=:& S_1 + S_2 + S_3. 
\end{eqnarray*}
From  \eqref{det2} we have
\begin{eqnarray*}
S_1 = \det \begin{pmatrix} B + 2K & 0 \\ 0& \frac{1}{2}\end{pmatrix} = \frac{1}{2} \det ( B + 2K)  = \frac{1}{2} (|\textbf{x}|^4 + a^2 \textbf{t}^2)^{n-1} (3|\textbf{x}|^4 + a^2 \textbf{t}^2).
\end{eqnarray*}
Using $\textrm{rank} ~K =1$ we get
\begin{eqnarray*}
\det \begin{pmatrix} e_1 \\ \vdots \\ e_{j-1} \\ r_j \\ e_{j+1} \\ \vdots \\ e_{2n} \end{pmatrix} = \det \begin{pmatrix} b_1 + 2 k_1 \\ \vdots \\ b_{j-1} + 2 k_{j-1} \\ \frac{-(\beta+4)|\textbf{x}|^4}{|\textbf{x}|^4 +\textbf{t}^2} k_j \\ b_{j+1} + 2 k_{j+1} \\ \vdots \\ b_{2n} + 2 k_{2n} \end{pmatrix} = - \frac{(\beta+4)|\textbf{x}|^4}{|\textbf{x}|^4 + \textbf{t}^2}  \det  \begin{pmatrix} b_1 \\ \vdots \\ b_{j-1} \\ k_j \\ b_{j+1} \\ \vdots \\ b_{2n} \end{pmatrix}. 
\end{eqnarray*}
Therefore, 
\begin{eqnarray*}
S_2 &=& - \frac{1}{2} \left(\frac{(\beta+4)|\textbf{\textbf{x}}|^4}{|\textbf{\textbf{x}}|^4 + \textbf{\textbf{t}}^2} \right) |\textbf{x}|^4 (|\textbf{x}|^4 + a^2 \textbf{t}^2)^{n-1}.
\end{eqnarray*}
Finally, 
\begin{eqnarray*}
S_3 &=& \det \begin{pmatrix} B + 2K & 0 \\ * & - \frac{\beta+4}{|\textbf{x}|^4 + \textbf{t}^2}\frac{\textbf{t}^2}{4} \end{pmatrix} = - \frac{\beta+4}{|\textbf{x}|^4+\textbf{t}^2} \frac{\textbf{t}^2}{4} \det (B+2K) 
\\
&=& -\frac{\beta+4}{|\textbf{x}|^4 + \textbf{t}^2} \frac{\textbf{t}^2}{4} (|\textbf{x}|^4 +a^2 \textbf{t}^2)^{n-1} (3|\textbf{x}|^4 +a^2 \textbf{t}^2).
\end{eqnarray*}
Adding all these terms together, we get
\begin{eqnarray*}
\det (E+R)= p (|\textbf{x}|^4 + a^2 \textbf{t}^2)~ q (|\textbf{x}|^4 + \textbf{t}^2) ~f(\textbf{x},\textbf{t}),
\end{eqnarray*}
where $p(r) = c_p r^{m_1}$, $q(r) = r^{m_2}$ for some $m_1, m_2, c_p \in \mathbb{R}$ and 
\begin{eqnarray*}
f(\textbf{x},\textbf{t}) = 2 (\beta+1) |\textbf{x}|^8 + ( 3 (\beta+2) - 2a^2) |\textbf{x}|^4 \textbf{t}^2 + (\beta+2)a^2 \textbf{t}^4. 
\end{eqnarray*} 
The proof is complete.
\end{proof}
Now, we should determine when the determinant of $H(x,t,y,s)$ can be zero for some values $(x,t,y,s)$ with $\rho\left( (x,t)\cdot (y,s)^{-1}\right) \sim 1$. Furthermore, to determine the type of folds in the degenerate cases, it is crucial to know the shape of the factorization. 
\begin{lem}\label{factor} There are nonzero constants $\gamma, c,c_1,c_2,c_3$ with $c_1 \neq c_2$ and $c_3 >0$ that are determined by $\beta$ and $a$ such that:
\
\begin{itemize}
\item Case 1:
\begin{itemize}
\item[$\cdot$] If $\beta \in (-1,0) \cup (0,\infty)$ and ${a^2} < C_{\beta}$, then $f(x,t) > 0$.
\item[$\cdot$] If $\beta \in (-1,0) \cup (0,\infty)$ and ${a^2} = C_{\beta}$, then $f(x,t) = \gamma ( |x|^2 - c t^2)^2$.
\item[$\cdot$] If $\beta \in (-1,0) \cup (0,\infty)$ and ${a^2} > C_{\beta}$, then $f(x,t) = \gamma ( |x|^2 - c_1 t)(|x|^2 + c_1 t)(|x|^2 - c_2 t)(|x|^2 + c_2 t)$.
\end{itemize}
\item Case 2:
\begin{itemize}
\item[$\cdot$] If $\beta \in (-2,-1)$, then $f(x,t) = \gamma (|x|^2- c_1 t) (|x|^2 + c_1 t) (|x|^4 + c_3 t^2)$.
\end{itemize}
\item Case 3:
\begin{itemize}
\item[$\cdot$] If $\beta \in (\infty, -2)$, then $ f(x,t) < 0$.
\end{itemize}
\end{itemize}
\end{lem}
\begin{proof}
Let $g(y,s) = 2(\beta+1) y^2 + (3(\beta+2)-2a^2) ys + (\beta+2) a^2 s^2$. Then $f(x,t) = g(|x|^4, t^2)$. Suppose $\beta \in (-1,0) \cup (0,\infty)$. 
First, we see that $f(x,t) >0$ for $3(\beta +2) - 2a^2 > 0$. Secondly, we have $f(x,t) >0 $ if
\begin{eqnarray*}
\Delta := 4a^{4} - 4 (\beta+2) (2\beta + 5) a^2 + 9 (\beta +2)^2 < 0.
\end{eqnarray*}
This holds if and only if
\begin{eqnarray*}
C_{\beta}^{-} < a^2 < C_{\beta}^{+},
\end{eqnarray*} 
where
\begin{eqnarray*}
C_{\beta}^{\pm} = \frac{\beta +2}{2} \left(2\beta +5 \pm \sqrt{(2\beta  + 5)^2 - 9 }\right) .
\end{eqnarray*}
Observe that
\begin{eqnarray*}
C_{\beta}^{-} = \frac{(\beta+2)}{2} (2\beta + 5 - \sqrt{(2\beta +5)^2 -9}) & = & \frac{(\beta+2)}{2} (2\beta + 5 - \sqrt{(2\beta+2)(2\beta+8)}
\\
& < & \frac{(\beta+2)}{2} (2 \beta+5 - \sqrt{(2\beta+2)^2}) = \frac{3 (\beta+2)}{2}.
\end{eqnarray*}
We can combine the above two conditions as $g(y,s) >0$ for $a^2 < C_{\beta}^{+}$. For $a^2 = C_{\beta}$, we have $g(y,s) = \gamma(y-cs)^2 $ for some $c >0$. For $a^2 > C_{\beta}$, we have $g(y,s) = \gamma (y-c_1 s) (y-c_2 s)$ for some $c_1, c_2 >0$ since $2(\beta+1) \cdot (\beta+2) a^2 > 0$. 
\

Finally, if $\beta \in (-2,-1),$ then $2(\beta+1)(\beta+2) a^2 < 0$, and so $g(y,s) = \gamma (y-c_1 s) (y+c_2 s)$. If $\beta \in (-\infty, -2)$, then $2(\beta+1)<0, ~ 3(\beta+2)-2a^2 <0$ and $\beta+2 <0$. Thus $g(y,s) < 0$. This completes the proof.
\end{proof}

\begin{lem}\label{rank} Let $L_1(x,t,y,s)$ be the upper left $(2n) \times (2n)$ block matrix of $L(x,t,y,s)$ and suppose that $\left( x,t,y,s\right)$ is contained in $S$. If $ \beta \neq -4$, then 
\begin{eqnarray*}
\det L_1 (x,t,y,s) \neq 0.
\end{eqnarray*}
\end{lem}

\begin{proof}
For simplicity, set $(z,w) := (x,t)\cdot (y,s)^{-1}$. In view of \eqref{BKE} and \eqref{sowe}, except the nonzero common facts,  we only need to check that the determinant of 
\begin{eqnarray*}
M(z,w)=  \begin{pmatrix} |z|^2 I + a w J + 2 z \cdot z^T  - ( \beta+4) \frac{|z|^4}{|z|^4 + w^2 } x \cdot z^T \end{pmatrix},
\end{eqnarray*}
is nonzero for $(z,w) \neq (0,0)$. This determinant can be calculated in the same way as the determinant of $L$ by using \eqref{detB} and \eqref{det2}. We find
\begin{equation}\label{det10}
\begin{split}
\det(M(z,w)) =& (|z|^4 + a^2 w^2)^n + (|z|^4 + a^2 w^2 )^{n-1} |z|^4 \left( 2 - (\beta+4) \frac{|z|^4}{|z|^4 + w^2} \right)
\\
=& \frac{(|z|^4+a^2 w^2)^{n-1}}{|z|^4 + w^2} \left[  - (\beta+1) |z|^8 + (a^2 + 3) |z|^4 w^2 + a^2 w^4 \right]. 
\end{split}
\end{equation}
Notice that $(z,w)$ is in $S$ and satisfies
\begin{eqnarray}\label{zero}
2(\beta+1)|z|^8 + (3 (\beta+2) - 2a^2) |z|^4 w^2 + (\beta+2) a^2 w^4 = 0. 
\end{eqnarray} 
From \eqref{det10} and \eqref{zero} we get
\begin{eqnarray*}
\det(M(z,w)) =  \frac{(|z|^4+a^2 w^2)^{n-1}}{|z|^4 + w^2} \frac{w^2}{2}  (\beta+4)\left[ 3|z|^4 + a^2 w^2 \right].
\end{eqnarray*}
If $w=0$, then $z$ becomes zero in \eqref{zero}. Because $(z,w) \neq (0,0)$, $w$ should be nonzero. Thus $\det (M(z,w)) \neq 0$. The Lemma is proved.
\end{proof}

We are now ready to prove our first main theorems by studying the canonical relation \eqref{relation1} associated to the phase~$\Phi$,
\begin{eqnarray*}
C_{\Phi} = \{ \left( (x,t) , \Phi_{(x,t)}, (y,s), - \Phi_{(y,s)} \right) \} \subset T^{*} (\mathbb{R}^{2n+1}) \times T^{*} (\mathbb{R}^{2n+1}),
\end{eqnarray*}
and the associated projection maps $ \pi_L : C_{\Phi} \rightarrow T^{*} (\mathbb{R}^{2n+1}) \quad \textrm{and} \quad \pi_R : C_{\Phi} \rightarrow T^{*}(\mathbb{R}^{2n+1}).$

\begin{proof}[Proof of Proposition \ref{LA} Proposition \ref{LB}]

Let
\begin{eqnarray*}
S = \{ (x,t,y,s) : \det H (x,t,y,s) =0 \}.
\end{eqnarray*} 
In view of Proposition \ref{degenerate}, it is enough to show that on the hypersurface $S$, 
\begin{enumerate}
\item If $\beta \in (-2,-1)$ or  $\beta \in (-1,0) \cup (0,\infty)$ and ${a^2} > C_{\beta}$, then both projections $\pi_L$ and $\pi_R$ have 1-type folds singularities. 
\item If $\beta \in (-1,0) \cup (0, \infty)$ and ${a^2} = C_{\beta}$, then both $\pi_L$ and $\pi_R$ have 2-type folds singularities.
\end{enumerate}
We will only prove  (1). The second case can be proved in the same way, the only difference is the form of factorizations in Lemma \ref{factor} which determine the order of types. We need to show that on the hypersurface $S$, both $\pi_L$ and $\pi_R$ have 1-type folds singularities. Rcall from Lemma \ref{det} that $S$ is a subset of $\mathbb{R}^{2n+1}$ consisting of $(x,t,y,s) \in \mathbb{R}^{2(2n+1)}$ such that
\[ F\left((x,t)\cdot (y,s)^{-1}\right) = F\left( x-y, s-t + 2a x^T J y\right) =0 \quad \textrm{and}  \quad \rho\left( (x,t)\cdot (y,s)^{-1}\right) \sim 1. \]
From the form of $F$ and the fact that $((x,t) \cdot (y,s)^{-1}) \neq 0$, we have
\[
S= \{ (x,t,y,s) \in \mathbb{R}^{2(2n+1)} ~ | ~ f \left( x-y, s- t + 2ax^T J y\right) = 0, \quad \rho\left( (x,t)\cdot (y,s)^{-1} \right) \sim 1 \}.
\]
From Theorem \ref{factor}, we have
\begin{eqnarray*}
f(x,t) = \gamma (|x|^2 - c_1 t ) (|x|^2 + c_1 t) (|x|^2 - c_2 t )(|x|^2 + c_2 t).
\end{eqnarray*}
for some two different constants $c_1, c_2 >0$. 
\

Note that Lemma \ref{rank} implies the condition (1) of Definition \ref{def} is satisfied. Therefore, it is enough to show the second condition, i.e., at each point $P_0 \in S$ the determinant of $Df$ vanishes with order 1 in each null direction of $d\pi_L$ and $d\pi_R$ at $P_0$.
 Fix a point $P_0 = (x, t, y, s) \in \mathbb{R}^{2n+1} \times \mathbb{R}^{2n+1}$ and assume that $P_0$ is contained in 
\begin{eqnarray*}
S_1 =: \{ (x,t,y,s) \in \mathbb{R}^{2(2n+1)} ~ | ~ |x-y|^2 - c_1 (s - t + 2a x^T J y) = 0 \}.
\end{eqnarray*}
We may identify $C_{\Phi} = \{ \left( (x,t), \Phi_{(x,t)}, (y,s), - \Phi_{(y,s)} \right) \}$ with an open set in $\mathbb{R}^{(2n+1)} \times \mathbb{R}^{(2n+1)}$ by the diffeomorphsim $\psi : \mathbb{R}^{(2n+1)} \times \mathbb{R}^{(2n+1)} \rightarrow S$ given by
\begin{eqnarray*}
\psi (x,t,y,s) = \left( (x,t), \Phi_{(x,t)}, (y,s), - \Phi_{(y,s)} \right).
\end{eqnarray*}
Let $v_L \in \mathbb{R}^{2(2n+1)}$ be a null direction of $d \pi_L$ at $P_0$, i.e.,
\begin{eqnarray*}
\begin{pmatrix} I & 0 \\ \frac{\partial^2 \Phi}{\partial_{(x,t)} \partial_{(x,t)} } & \frac{\partial^2 \Phi}{\partial_{(y,s)}\partial_{(x,t)}}  \end{pmatrix} v_L^T = 0.
\end{eqnarray*}
Thus, $v_L$ is of the form $v_L = (0,0,z,w)$ with $ w \in \mathbb{R}^{2n}$ and $s \in \mathbb{R}$  such that
\begin{eqnarray}\label{200}
\frac{\partial^2 \Phi}{\partial_{(y,s)}\partial_{(x,t)}} \begin{pmatrix} z^T \\ w \end{pmatrix} = 0.
\end{eqnarray} 
To check that $\det H(x,t,y,s)$ vanishes of order 1 in the direction $v_L$, it is enough to show that $v_L$ is not orthogonal to the gradient vector $v_g$ of $\det H(x,t,y,s)$ at $P_0$.
By a direct calculation we see that the gradient vector $v_g$ is equal to 
\begin{eqnarray*}
\left. D_{(x,t),(y,s)} \Phi \left( (x,t) \cdot (y,s)^{-1}\right)\right|_p = \left( 2(x-y) - 2 a c_{1} a J y, ~ - c_{1}, ~ -2 (x-y) - 2 a c_{1} x^{T} J,~ c_{1} \right)
\end{eqnarray*}
Suppose with a view to contradiction that $v_L$ and $v_g$ are orthogonal.
 It means that 
\begin{eqnarray}\label{OG}
-2 (x-y) \cdot z - 2 a c_{1} x^T J \cdot z + c_{1} w = 0.
\end{eqnarray}
From \eqref{35}, we have
\begin{eqnarray}\label{201}
\frac{\partial^2 \Phi}{\partial_{(y,s)}\partial_{(x,t)}}\begin{pmatrix} z^{T} \\ w \end{pmatrix} = A_a (y) \left[ (\partial_i \partial_j \Phi) - 2a (\partial_{2n+1} \Phi) \begin{pmatrix} J & 0 \\ 0 & 0 \end{pmatrix} \right](\textbf{x},\textbf{t})~ A_a (x)^T \cdot \begin{pmatrix} z^T \\ w \end{pmatrix}.
\end{eqnarray}
A simple calculation shows that
\begin{eqnarray*}
A_a (x)^T \cdot \begin{pmatrix} z^T \\ w \end{pmatrix} &=&
\begin{pmatrix} 1 & 0 & 0 & \cdots & 0 & 0 \\ 0 & \ddots & 0 & \cdots & 0 &0 \\
0&0&1  & \cdots &0&\vdots \\  0 & 0 &0& \ddots & 1 & 0 \\ 2a x_{n+1} & \cdots  & -2a x_1 & \cdots & -2a x_n & 1 
\end{pmatrix} 
\begin{pmatrix}
z_1 \\ z_2 \\ \vdots \\ z_{2n} \\ w 
\end{pmatrix}
\\
&=& \begin{pmatrix} z_1, & z_2,  & \cdots, & z_{2n}, & 2a(x_{n+1} z_1 + \cdots + x_{2n} z_n - x_1 z_{n+1} - \cdots - x_n z_{2n} ) + w
\end{pmatrix}^T.
\end{eqnarray*}
On the other hand, from the orthogonal assumption \eqref{OG} we get
\begin{eqnarray*}
2a (x_{n+1} z_1 + \cdots - x_n z_{2n} ) + w = \frac{2 (x-y) \cdot z } {c_{1}}.
\end{eqnarray*}
Thus, 
\begin{eqnarray*}
A_a (x)^{T} \cdot \begin{pmatrix} z^T \\ w \end{pmatrix} = \begin{pmatrix} z_1, & z_2, & \cdots, & z_{2n}, & \frac{2(x-y) \cdot z }{c_{1}} \end{pmatrix}^T.
\end{eqnarray*}
Recall that
\begin{eqnarray*}
\left[ (\partial_i \partial_j \Phi) - 2 a (\partial_{2n+1} \Phi) \begin{pmatrix} J & 0 \\ 0&0 \end{pmatrix} \right] (x,t)= (\beta+4) \begin{pmatrix} |x|^4 x_1^2 & \cdots & |x|^4 x_1 x_n & |x|^2 x_1 \frac{t}{2} 
\\
\vdots & \ddots & \vdots & \vdots 
\\
|x|^4 x_n x_1 & \cdots & |x|^4 x_n^2 & |x|^2 x_n \frac{t}{2} 
\\
|x|^2 \frac{t}{2} x_1 & \cdots & |x|^2 \frac{t}{2} x_n & \frac{t^2}{4} 
\end{pmatrix} - (|x|^4 + t^2 ) \begin{pmatrix} J & 0 \\ 0 & \frac{1}{2} \end{pmatrix}. 
\end{eqnarray*}
Substituting $x-y$ for $x$ and  $ t-s + 2a x^T J y = \frac{|x-y|^2}{c_{1}}$ for $t$, where the equality holds since the point $P_0$ is on the surface $S_1$. Then, from $(2n+1)$-th equality in \eqref{200} with \eqref{201}, we have
\begin{eqnarray*}
(\beta + 4) \left[ |x-y|^2 \cdot \frac{1}{2} \frac{|x-y|^2}{c_{\beta,1}} (x-y) \cdot z + \frac{|x-y|^4}{c_{\beta,1}^2} \cdot \frac{2}{c_{\beta,1}} (x-y) \cdot z \right] - \frac{1}{2} ( |x-y|^4 + \frac{|x-y|^4}{c_{\beta,1}^2} )\frac{2}{c_{\beta,1}} (x-y) \cdot z = 0.
\end{eqnarray*}
Rearranging it, we obtain
\begin{eqnarray*}
\left[ \frac{\beta+2}{2 c_{\beta,1}} + \frac{1}{c_{\beta,1}^3}\right] |x-y|^4~ (x-y) \cdot z = 0. 
\end{eqnarray*}
Thus $ (x-y) \cdot z =0$, and hence 
\begin{eqnarray*}
A_a (x)^T \cdot \begin{pmatrix} z \\ w \end{pmatrix} = \begin{pmatrix} z_1, z_2, \cdots, z_{2n}, 0\end{pmatrix}^{T} \quad \textrm{and} \quad L_1 (x,t,y,s) \cdot \begin{pmatrix} (z_1,z_2,\cdots, z_{2n})\end{pmatrix}^{T} =0.  
\end{eqnarray*}
Now from $\det L_1 \neq 0$ in Lemma \ref{rank} we have $z=0$ and so $w=0$ from \eqref{OG}. This is a contradiction since $v_L$ should be a nonzero direction vector. Therefore $v_L$ and $v_R$ can not be orthogonal.
\

Now we shall prove the same conclusion for  $d\pi_R$ without repeating the calculations.  
Note that the above argument  for $d \pi_L$ is exactly to show that there is no nontrivial solution $(z,w)$ of the system of equation $S({a,x,y}) $: 

\begin{eqnarray*}
(\frac{\partial^2}{\partial x_i \partial y_j} \Phi) \begin{pmatrix} z^T \\ w \end{pmatrix} = A_a (y) \left[ (\partial_i \partial_j \Phi) - 2a (\partial_{2n+1} \Phi) \begin{pmatrix} J & 0 \\ 0 & 0 \end{pmatrix} \right] A_a (x)^T \cdot \begin{pmatrix} z^T \\ w \end{pmatrix} = 0,
\end{eqnarray*}
and
\begin{eqnarray*}
(-2(x-y) - 2 a c_{\beta,1} x^T J, c_{\beta,1}) \cdot (z,w) = 0.
\end{eqnarray*}
On the other hand, to show the folding type condition for the projection $\pi_R$, it is enough to show that there is no nontrivial solution $v_R = ( z_0, w_0 , 0, 0 )$ which satisfies the system of equations :
\begin{eqnarray*}
(\frac{\partial^2}{\partial y_i \partial x_j} \Phi)\begin{pmatrix} z_0^T \\ w_0 \end{pmatrix} = A_a (x) \left[ (\partial_i \partial_j \Phi) + 2a (\partial_{2n+1} \Phi) \begin{pmatrix}  J & 0 \\ 0 & 0 \end{pmatrix} \right] A_a (y)^T \cdot \begin{pmatrix} z_0^T \\ w_0 \end{pmatrix} = 0,
\end{eqnarray*}
and
\begin{eqnarray*}
\left( 2(x-y) + 2a c_{\beta,1} y^T J ,~ - c_{\beta,1} \right) \cdot (z_0, w_0) = 0.
\end{eqnarray*}
Because $A_{-a}(-x) = A_a (x)$ and $A_{-a}(-y) = A_a (y)$, the above system can be written as follows.
\begin{eqnarray*}
(\frac{\partial^2}{\partial y_i \partial x_j} \Phi)\begin{pmatrix} z_0^T \\ w_0 \end{pmatrix} = A_{-a} (-x) \left[ (\partial_i \partial_j \Phi) - 2(-a) (\partial_{2n+1} \Phi) \begin{pmatrix}  J & 0 \\ 0 & 0 \end{pmatrix} \right] A_{-a} (-y)^T \cdot \begin{pmatrix} z_0^T \\ w_0 \end{pmatrix} = 0,
\end{eqnarray*}
and
\begin{eqnarray*}
\left( - 2((-y)-(-x)) - 2(-a) c_{\beta,1} (-y)^T J , ~ c_{\beta,1} \right) \cdot (z_0, w_0) = 0.
\end{eqnarray*}
We now see that $( z_0, w_0)$ satisfies the system $S({-a, -y, -x})$. Since the above argument for proving nonexistence of nontrivial solution of $S(a,x,y)$ does not depend on specific values of $a$, $x$ and $y$, the same conclusion holds for the system $S(-a,-y,-x)$. This completes the proof.
\end{proof}

\begin{rem}On $\mathbb{R}^n$, the oscillating kernel is of the form $|x|^{-\gamma} e^{i|x|^{\beta}}$ with $\beta \neq 0$. The behavior for the phases $|x|^{\beta}$   depends only on whether $ \beta \neq 1$ or $\beta =1$.  Precisely, for $\beta \neq 1$, we have $\det \left( \frac{\partial^2}{\partial x \partial y} |x-y|^{\beta} \right) \neq 0 $ for any $(x,y)$ with $x\neq y$, but  $\det \left( \frac{\partial^2}{\partial x \partial y} |x-y| \right) = 0$ for  any $(x,y)$ with $x \neq y$ and this case correspond to Bochner-Riesz means operators, which still remains as a conjecture. On hand, the phase $\rho((x,t)\cdot (y,s)^{-1})^{\beta}$ has fold of the highest order type when $\beta=1$ or $\beta = 2$, which also remains open in this paper. In order to establish the sharp $L^2$ estimate for these cases, we would need to improve the current theory of oscillatory integral estimates for degenerate phases to higher orders (see \cite{cu, GR, GR2}). 
\end{rem}
\begin{rem} We note that from Lemma \ref{rank} and Case 3 of Lemma \ref{factor}, 
\begin{eqnarray}\label{L2}
\| L_{A_j} \|_{L^2 \rightarrow L^2} + \| L_{B_j}\|_{L^2 \rightarrow L^2} \lesssim 2^{j (\alpha - n \beta)} 
\end{eqnarray}
holds for all cases. It will be sufficient to use this weaker bound for the Hardy spaces estimates in Section 5. 
\end{rem}

\

\section{Hardy spaces on the Heisenberg groups}
In this section we recall some properties of Hardy spaces on the Heisenberg group. We refer Coifman-Weiss \cite{Co} and Folland-Stein \cite{Fo} for the details. From now on, we shall write $\rho(x)$ (resp., $x\cdot y$) just as $|x|$ (resp., $xy$) for  notational convenience. It is known that $|x \cdot y | \leq |x| + |y|$ holds for all $x, y \in \mathbb{H}^n_a$ (see \cite[p. 688]{Lin}). 
\

The left-invariant vector fields on $\mathbb{H}^n_a$ is spanned by $T=\frac{\partial}{\partial t}$ and $X_j = \frac{\partial}{\partial {x_j}} + 2ax_{n+j} \frac{\partial}{\partial t},\quad X_{j+n}=\frac{\partial}{\partial{x_{j+n}}}-2ax_{n}\frac{ \partial}{\partial t}$, $1\leq j \leq n$. Let $Y_j = X_j $ for $1\leq j \leq 2n$ and $Y_{2n+1} = T$. We say that the right-invariant differential operator $Y^I = Y_1^{i_1}\cdots Y_{2n+1}^{i_{2n+1}}$ has homogeneous degree $d(I) = i_1 + i_2 + \cdots + i_{2n} + 2 i_{2n+1}$. For $a \in \bar{\mathbb{N}}$, we define $\mathcal{P}_a$ to be the set of all homogeneous polynomials of degree $a$.

Suppose that $x \in \mathbb{H}^n_a$, $ a \in \bar{\mathbb{N}}$, and $f$ is a function whose distributional derivatives $Y^I f$ are continuous in a neighborhood of $x$ for $d(I) \leq a$. The homogeneous right $Taylor~polynomial $ of $f$ at $x$ of degree $a$ is the unique $P_{f,x} \in \mathcal{P}_a$ such that $Y^{I}P_{f,x}(0) = Y^{I} f(x)$ for $d(I) \leq a$. 

\begin{prop}[\cite{Fo}]\label{appro} Suppose that $f \in C^{k+1}$, $T \in \mathcal{S}'$, and $P_{f,x}(y) = \sum_{d(I)\leq k} a_I(x) \eta^{I}(y)$ is the right Taylor polynomial of $f$ at $x$ of homogeneous degree $k$. Then $a_I$ is a linear combination of the $Y^{J} f$ for $d(J) \leq k$,
\begin{eqnarray}\label{formula}
|f(y x)-P_{f,x}(y)| \leq C_k |y|^{k+1} \sup_{\substack{ d(I)=k+1\\|z|\leq b^{k+1}|y|}} |Y^{I} f(z x)|.
\end{eqnarray} 
\end{prop}

We will use some properties for $H^p$ functions including the atomic decomposition and the molecular characterization. For $0 < p \leq 1 \leq q \leq \infty, p \neq q, s \in \mathbb{Z}$ and $s \geq [(2n+2)(1/p-1)]$, we say that the triple $(p,q,s)$ is admissible.
\begin{defn}
For an admissible triple $(p,q,s)$, we define $(p,q,s)$-atom centered at $x_0$ as a function $a \in L^{q} (\mathbb{H}^{n})$ supported on a ball $B \subset \mathbb{H}^n_a$ with center $x_0$ in such way that
\begin{enumerate}
\item[(i)] $\|a\|_q \leq |B|^{1/q-1/p}.$
\item[(ii)] $\int_{\mathbb{H}^n} a (x) P(x) dx = 0$ for all $P \in \mathcal{P}_s$.
\end{enumerate} 
\end{defn}
Later, we will choose $q=2$ to use the $L^2$ boundedness \eqref{L2} obtained in Section 3. 
\begin{prop}[Atomic decomposition in $\mathbb{H}^p$; see \cite{Co}]Let $(p,q,s)$ be an admissible triple. Then any $f$ in $H^p$ can be represented as a linear combination of $(p,q,s)$-atoms,
 \[ f = \sum_{i=1}^{\infty} \lambda_i f_i, \qquad\lambda_i \in \mathbb{C},\]
where the $f_i$ are $(p,q,s)$-atoms and the sum converges in $H^p$. Moreover, $\|f\|_{H^p}^p \sim \inf \{\sum_{i=1}^{\infty}|\lambda_i|^p : \sum \lambda_i f_i \textrm{ is a decomposition of f into (p,q,s)-atoms} \}.$
\end{prop}
 For an admissible triple $(p,q,s)$, we choose an arbitrary real number $\epsilon > \max \{ s/(2n+2), 1/p-1\}$. Then we call $(p,q,s,\epsilon)$ an admissible quadruple.  Now we introduce the molecules.
\begin{defn} Let $(p,q,s, \epsilon)$ be an admissible quadruple. We set
\begin{eqnarray}\label{ab}
a = 1- 1/p + \epsilon, \quad b = 1-1/q + \epsilon.
\end{eqnarray}
A $(p,q,s,\epsilon)$-molecule centered at $x_0$ is a function $M \in L^{q}(\mathbb{H}^n)$ such that
\begin{enumerate}
\item $M(x) \cdot |x_0^{-1} x |^{(2n+2)b} \in L^q (\mathbb{H}^n). $
\item $\| M \|_q^{a/b} \cdot \|M(x) \cdot |x_0^{-1} x|^{(2n+2)b} \|_q^{1-a/b} \equiv \mathcal{N}(M) < \infty.$
\item $\int_{\mathbb{H}^n} M(x) P(x) dx  = 0 $ for every $P \in \mathcal{P}_s$.
\end{enumerate}
\end{defn}
\begin{thm}\label{mole}\mbox{~}
\begin{enumerate}
\item  Every $(p,q, s')$-atom $f$ is a $(p,q,s,\epsilon)$-molecule for any $\epsilon > \max\{s/(2n+2), 1/p-1\}, s \leq s' $ and $\mathcal{N} (f) \leq C_1$, where the constant $C_1$ is independent of the atom.
\item  Every $(p,q,s,\epsilon)$-molecule $M$ is in $H^p$ and $\|M\|_{H^p} \leq C_2 \mathcal{N} (M)$, where the constant $C_2$ is independent of the molecule.
\end{enumerate}
\end{thm}
Thanks to this Theorem, in order to verify that $T$ is bounded on $H^p$ it is enough to show that, for all $p$-atoms $f$, the function   $T f$ is a $p$-molecule and $\mathcal{N}(Tf) \leq C$ for some constant $C$ independent $\textrm{of}~f$.

\

\section{$H^p$ estimates}
We start with a lemma which will be useful in the proofs of the sequel.
\begin{lem}\label{sum}~
\begin{enumerate}
\item Suppose that $ d<0$, $c+d < 0$ and $B > 1$. Then
\begin{eqnarray*}
\sum_{j=1}^{\infty} 2^{cj} \min \{ 1, B 2^{dj}\} \lesssim 1 + (\log B)B^{-\frac{c}{d}}.
\end{eqnarray*}
\item Suppose that $c< 0$, $d>0$ and $B< 1$. Then 
\begin{eqnarray*}
\sum_{j=1}^{\infty} 2^{cj} \min \{ 1, B2^{dj}\} \lesssim B + |\log B|  B^{-\frac{c}{d}}.
\end{eqnarray*}
\end{enumerate}
\end{lem}
\begin{proof} 
 Set $K = \sum_{j=1}^{\infty} 2^{cj} \min \{ 1, B 2^{dj}\}$. Then,
\[ K = \sum_{ B 2^{dj} \leq 1} 2^{(c+d)j } + \sum_{ B 2^{dj} > 1} 2^{cj}.\]
A straighforward calculation gives the bound for $K$. Suppose that $ d<0, c+d>0$ and $B>1$. Then
\begin{itemize}
\item[-] $K \lesssim 1  $ \quad \quad \quad \quad{for} $ c<0$, 
\item[-] $K \lesssim \log B $\quad \qquad{for}  $c = 0,$ 
\item[-] $K \lesssim B^{-\frac{c}{d}}$ \quad \qquad for $c>0$.
\end{itemize}
In any case we see that $K \lesssim1 + (\log B) B^{-\frac{c}{d}}$.
\noindent Suppose now that $c<0, d>0$ and $B<1$. Then
\begin{itemize}
\item[-] $K \lesssim B$ \qquad \qquad for $ c+d <0,$
\item[-] $K \lesssim \log B \cdot B$ \quad \quad~for $ c+d =0,$ 
\item[-]  $K \lesssim B^{-\frac{c}{d}}$ \qquad \quad for $c+d >0$.
\end{itemize}
\noindent In any case we have $K \lesssim B + |\log B| B^{-\frac{c}{d}}$. The Lemma is proved.
\end{proof}
\begin{thm}\label{K} Assume $p\leq 1$ and
 $(\frac{1}{p} -1)(2n+2) \beta + \alpha < 0$. Then $T_{K_{\alpha,\beta}}$ is bounded on $H^p$.
\end{thm}
\begin{proof}
From the decompostion of kernel \eqref{decom1}, we have
\begin{eqnarray*}
\| K_{\alpha,\beta} * f \|_{H^p}^p \leq \sum_{j \geq 1} \| K_{\alpha,\beta}^{j} * f \|_{H^p}^p.
\end{eqnarray*}
We shall bound the norm $\|K_{\alpha,\beta}^{j} * f\|_{H^p}$ for each $j \in \mathbb{N}$  by some constant multiple of $\| f \|_{H^p}$. Notice that $K_{j} (x,t) = \rho(x,t)^{-(2n+2+\alpha)} e^{i \rho(x,t)^{-\beta}} \chi (2^j \rho(x,t)).$ From the atomic decomposition for $H^p$ space, it is enough to establish the estimate for any atom $f$ supported on $B(0,R)$ with some $R>0$ such that
\begin{equation}\label{32}
\begin{split} \textrm{-}~&\| f \|_{L^2} \leq R^{(2n+2)(\frac{1}{2} - \frac{1}{p})}, \quad \qquad \qquad  \qquad \qquad \qquad
\\
\textrm{-}~&\int f(x) x^{\alpha} dx = 0, \quad \textrm{for all} ~ |\alpha| \leq s = [ (2n+2)(\frac{1}{p} -1 )].\quad \qquad \qquad \qquad \qquad \qquad
\end{split}
\end{equation}
In view of part (2) of  Theorem \ref{mole} it suffices to bound $\mathcal{N}(K_j* f)$. For an admissible quadruple, we choose an $\epsilon > \max\{ \frac{s}{2n+2}, \frac{1}{p} -1 \}= \frac{1}{p} -1$ and set $\epsilon = \frac{1}{p} -1 + \delta$ with some $\delta >0$. Then we have $ a = \delta $ and $ b = \frac{1}{p} - \frac{1}{2} + \delta ~\textrm{in} ~\eqref{ab}.$ We will choose $\delta$ sufficiently small later. Recall that $\mathcal{N}(K_j * f) = \| K_j * f \|_2^{a/b} \cdot \|K_j * f(x) \cdot |x|^{(2n+2)b} \|_2^{1-a/b}$. From the $L^2$ estimate \eqref{L2} we get  
\begin{eqnarray}\label{K1}
\|K_j * f\|_2 \lesssim 2^{j(\alpha - n \beta)} \| f\|_2.
\end{eqnarray}
We have
\begin{eqnarray*}
\| K_j * f(x) \cdot |x|^{(2n+2)b} \|_2^{2} &=& \int_{\mathbb{H}^n} |K_j * f (x) |^2 \cdot |x|^{2(2n+2)b} ~dx = I_1 + I_2, 
\end{eqnarray*}
where 
\[ I_1 = \int_{|x| \leq 2R} |K_j * f(x)|^2 \cdot |x|^{2(2n+2)b} ~dx \quad \textrm{and} \quad I_2 = \int_{|x| > 2R} |K_j * f (x)|^2 \cdot |x|^{2(2n+2)b} ~dx.\]
Then
\begin{equation}\label{K2}
\begin{split}
\sum_{j\geq1} \|K_j * f\|_{H^p}^p           \lesssim &\sum_{j\geq1} \mathcal{N}(K_j*f)^p
\\
       \lesssim& \sum_{j\geq1} \left(\|K_j * f \|_2^{a/b} \cdot (I_1^{1/2(1-a/b)} + I_2^{1/2(1-a/b)})\right)^{p}
\\
       \lesssim& \sum_{j\geq1} \|K_j * f\|_2^{pa/b} \cdot I_1^{p/2(1-a/b)} + \sum_{j\geq1} \|K_j * f\|_2^{pa/b} \cdot I_2^{p/2(1-a/b)}
\end{split}
\end{equation}
Set $S_1 = \sum_{j \geq 1} \|K_j * f\|_2^{pa/b} \cdot I_1^{p/2(1-a/b)}$ and $S_2 = \sum_{j \geq 1} \|K_j * f\|_2^{pa/b} \cdot I_2^{p/2(1-a/b)}$. Then it is enough to show that $S_1 \lesssim1 $ and $S_2 \lesssim 1$.
We use \eqref{L2} and   \eqref{32}   to bound $I_1$ as follows.
\begin{equation}\label{K3}
\begin{split}
I_1  \lesssim &  \int_{\mathbb{H}^n} | f * K_j (x)|^2 dx\cdot R^{2(2n+2)b } \lesssim  2^{2j(\alpha - n\beta)} \| f\|_2^{2} \cdot R^{2(2n+2)b }
\\
 \lesssim &~ 2^{2j(\alpha- n\beta)}  R^{2(2n+2)b } \cdot R^{(2n+2) (1 - 2/p)} 
\\
\lesssim  &~ 2^{2j(\alpha - (n+1/2) \beta)} R^{2(2n+2) \delta},
\end{split}
\end{equation}
where the last inequality comes from \eqref{32}. From \eqref{K1} and \eqref{K3} we have
\begin{eqnarray*}
\| K_j * f \|_2^{a/b} \cdot I_1^{1/2(1-a/b)}  &\lesssim& \left\{ 2^{j(\alpha-n \beta)} R^{(2n+2)(1/2-1/p)} \right\}^{a/b} \cdot \left\{ 2^{j(\alpha - n \beta)} \cdot R^{(2n+2) \delta} \right\}^{(1-a/b)}.
\\
&=& 2^{j(\alpha - n \beta)},
\end{eqnarray*}
where the equality comes from the calculation $(\frac{1}{2} - \frac{1}{p})\frac{a}{b} + a (1-\frac{a}{b}) = \frac{a}{b} ( \frac{1}{2}-\frac{1}{p} - a) + a = \frac{a}{b}(-b) + a = 0.$  Thus we have  $S_1 \lesssim \sum_{ j \geq 1} 2^{j(\alpha-n\beta)p} \lesssim 1.$
\

Now we  consider $I_2$ and $S_2$. We have $I_2 = 0$ for $R>1$ since the support of $K_j * f$ is contained in the subset $\{ x  : |x| \leq 1 + R \}$ which is a subset of $ \{ x  : |x| < 2R\}$ for $R > 1$. Thus we may only consider the case $R\leq 1$.  In the following integral expression
\[
(K_j * f) (x) = \int K_j (x y^{-1}) f (y) dy,
\]
We have $|xy^{-1}| \leq 2^{-j} $ and $|y| \leq R$. These imply $|x| \leq |xy^{-1}| + |y| \leq 2^{-j} + R$. It means that  $I_2 =0$ for $2^{-j} <R$. Thus we only need to consider $ j \in \mathbb{N}$ such that $2^{-j} \geq R$, for which we have $|x| \leq 2^{-j+1}$ for $x \in \textrm{Supp}(K_j * f)$. Then we get
\begin{eqnarray}\label{l_2}
I_2 = \int_{|x|>2R} | f * K_j (x)|^2 \cdot |x|^{2(2n+2)b} dx \lesssim \int_{|x|>2R} | f * K_j (x)|^2  dx \cdot 2^{-2(2n+2)bj}.
\end{eqnarray}
From Proposition \ref{appro}, for any $I \in \mathbb{N}_0$, there is a polynomial $P_j^{x}$ of degree $\leq I$ such that
\begin{equation}\label{ap}
\begin{split}
| K_j (xy^{-1}) - P_j^{x} (y) |  \lesssim & |y|^{I+1} \sup_{|\alpha| \leq I+1} | X^{\alpha} K_j (xy^{-1})|
\\
 \lesssim & |y|^{I+1} 2^{j(2n+2+\alpha)} 2^{j(\beta+1)(I+1)}.
\end{split}
\end{equation}
From \eqref{32} we get the identity for $0 \leq I \leq s,$
\[ K_j * f (x) = \int (K_j (x y^{-1}) - P_j^{x}(y)) f (y)dy. \] 
Note that $f(y)$ has support in $|y| \leq R$, then from \eqref{32} and \eqref{ap} we get 
\begin{eqnarray*}
|K_j * f (x)| &\lesssim &R^{I+1} 2^{j(2n+2+\alpha)} 2^{j(\beta+1)(I+1)} \int_{|y| \leq R} |f(y)| dy
\\
&\lesssim &R^{I+1} 2^{j(2n+2+\alpha)} 2^{j(\beta+1)(I+1)} R^{\frac{1}{2}(2n+2)} \|f\|_2
\\
&\lesssim & 2^{j(2n+2+\alpha)} (R 2^{j(\beta+1)})^{(I+1)} R^{(2n+2)(1-\frac{1}{p})}.
\end{eqnarray*}
Now we can estimate \eqref{l_2} as
\begin{eqnarray*}
I_2 &\lesssim& 2^{-2(2n+2)bj} 2^{-j(2n+2)} \left\{2^{j(2n+2+\alpha)} (R 2^{j(\beta+1)})^{(I+1)} R^{(2n+2)(1-\frac{1}{p})} \right\}^2.
\\
&=& 2^{2j \{(2n+2)(1-\frac{1}{p} - \delta) + \alpha\}} (R 2^{j(\beta+1)})^{2(I+1)} R^{2(2n+2)(1-\frac{1}{p})}.
\end{eqnarray*}
Here we may choose $I=0$ or $I=s$, which gives 
\begin{eqnarray*}
I_2 \lesssim 2^{2j \{(2n+2)(1-\frac{1}{p} -\delta) + \alpha\}} R^{2(2n+2)(1-\frac{1}{p})} \min \{ 1, (R 2^{j(\beta+1)})^{2(s+1)} \}.
\end{eqnarray*}
Now we have
\begin{equation}\label{I2}
\begin{split}
\|K_j*f\|_2^{a/b} \cdot I_2^{\frac{1}{2}(1-a/b)} \lesssim \{ 2^{j(\alpha -n\beta)} R^{(2n+2)(1/2-1/p)}\}^{a/b} &
\\
 \cdot \{2^{j \{(2n+2)(1-\frac{1}{p} -\delta) + \alpha\}} R^{(2n+2)(1-\frac{1}{p})} & \min\left( 1, (R 2^{j(\beta+1)})^{(s+1)}\right) \}^{(1-a/b)}.
\end{split}
\end{equation}
From $p\leq 1$ and $\alpha<0$  we have  $(2n+2)(1-\frac{1}{p} - \delta) + \alpha < 0 $. Thus, if $\min (1, (R2^{j(\beta+1)})^{s+1}) =1$ the exponent of $2^j$ is smaller than zero provided $a$ is small enough. Recall that $R \leq 1$. Then, using (2)  in Lemma \ref{sum} we get
\begin{eqnarray*}
\sum_{j \geq1} \|K_j * f \|_2^{pa/b} \cdot I_2^{\frac{p}{2}(1-a/b)} &\lesssim 
& R^{p \mu_{\delta}} + |\log R| \cdot R^{p \kappa_{\delta}},
\end{eqnarray*}
where 
\begin{eqnarray*}
R^{p\mu_{\delta}} &=& (R^{(2n+2)(1/2-1/p)})^{pa/b} \cdot (R^{(2n+2)(1-\frac{1}{p}) + (s+1)})^{p(1-a/b)}, 
\\
R^{p\kappa_{\delta}} &=&\left[ R^{-\frac{1}{\beta+1}[\alpha - (n+1/2) \beta] } R^{(2n+2) (1/2-1/p)} \right]^{p \delta/b} 
 \cdot \left[ R^{-\frac{1}{\beta+1} [ (2n+2)(1-1/p -\delta) + \alpha ]} R^{(2n+2)(1-1/p)} \right]^{p(1-\delta/b)}.
\end{eqnarray*}
Observe that
\[ \mu_{0} =  \{ (2n+2)(1-\frac{1}{p}) + (s+1)\} >0,
\]
and 
\[
\kappa_0 = - \frac{1}{\beta+1}[(2n+2)\beta (\frac{1}{p} -1) + \alpha] > 0.
\] 
Thus, for $\delta$ small enough, we have $\mu_{\delta}, \kappa_{\delta} >0$ and since $R \leq 1,$
\begin{eqnarray}\label{E3}
\sum_{j\geq 1} \|K_j * f\|_2^{pa/b} \cdot I_2^{\frac{p}{2} (1-a/b)} \lesssim R^{p \mu_{\delta}} + |\log R| \cdot R^{p \kappa_{\delta}} \leq 1.
\end{eqnarray}
We then conclude that $S_2 \lesssim 1$. The proof is complete.  
\end{proof}
We now consider $T_{L_{\alpha,\beta}}$. Observe that the oscillating term $e^{i \rho(x,t)^{\beta}}$ exhibits different behavior whether $0 < \beta <1$ or $\beta>1$. As $\rho$ goes to infinity, the oscillation becomes faint if for the case $0 < \beta <1$. In contrary, the oscillation grows to infinity for $ \beta >1$. Hence we deal with the two cases seperately.
\begin{thm}\label{L1}
Assume $0 < \beta < 1$ and $p \leq 1$ and $(\frac{1}{p} -1)(2n+2) \beta + \alpha < 0$. Then the operator $T_{L_{\alpha,\beta}}$ is bounded on $H^p$ space. 
\end{thm}
\begin{proof} From \eqref{decom2} we have
\begin{eqnarray}\label{Lsplit1}
\| L_{\alpha,\beta} * f \|_{H^p}^p \leq \sum_{j \geq 1} \| L_{\alpha,\beta}^{j} * f \|_{H^p}^p.
\end{eqnarray}
We now estimate each norm $\|L_{\alpha,\beta}^{j} * f\|_{H^p}$ by $\| f \|_{H^p}$. From the atomic decomposition for $H^p$ space, we may choose $f$ as an atom supported on $B(0,R)$ with some $R>0$, which satisfies
\begin{equation}\label{new}
\begin{split}
\textrm{-}~& \| f \|_{L^2} \leq R^{(2n+2)(\frac{1}{2} - \frac{1}{p})}, \quad \qquad \qquad  \qquad \qquad \qquad
\\
\textrm{-}~&\int f(x) x^{\alpha} dx = 0, \quad \textrm{for all} ~ |\alpha| \leq s = [ (2n+2)(\frac{1}{p} -1 )].\quad \qquad \qquad  \qquad \qquad \qquad
\end{split}
\end{equation}
From (b) in Theorem \ref{mole}, it suffices to estimate $\mathcal{N}(L_j* f)$. For an admissible quadruple $(p,q,s, \epsilon)$ we may choose any $\epsilon > \max\{ \frac{s}{2n+2}, \frac{1}{p} -1 \}= \frac{1}{p} -1$. Simply we let $\epsilon = \frac{1}{p} -1 + \delta$ with some $\delta >0$. Then we have $a = \delta $ and $
b = \frac{1}{p} - \frac{1}{2} + \delta.$ for \eqref{ab}. We will choose $\delta$ sufficiently small later. 
\

From \eqref{L2} we have
\[
\|L_j * f\|_2 \lesssim 2^{j(\alpha - n \beta)} \| f\|_2.
\] 
We have
\begin{equation}\label{28}
\begin{split}
\| L_j * f(x) \cdot |x|^{(2n+2)b} \|_2^{2} =& \int_{\mathbb{H}^n} |L_j * f (x) |^2 \cdot |x|^{2(2n+2)b} dx = I_1 + I_2,
\end{split}
\end{equation}
where 
\[ I_1 = \int_{|x|\leq 2R }|L_j * f(x)|^2 \cdot |x|^{2(2n+2)b} ~dx \quad \textrm{and} \quad I_2 = \int_{|x|> 2R} |L_j * f(x)|^2 \cdot |x|^{2(2n+2)b} ~ dx.
\]
Then,
\begin{equation}\label{Lsplit2}
\begin{split}
\sum_{j\geq1} \|L_j * f\|_{H^p}^p \lesssim &\sum_{j\geq1} \mathcal{N}(L_j*f)^p
\\
\lesssim& \sum_{j\geq1} \left(\|L_j * f \|_2^{a/b} \cdot (I_1^{1/2(1-a/b)} + I_2^{1/2(1-a/b)})\right)^{p}
\\
\lesssim& \sum_{j\geq1} \|L_j * f\|_2^{pa/b} \cdot I_1^{p/2(1-a/b)} + \sum_{j\geq1} \|L_j * f\|_2^{pa/b} \cdot I_2^{p/2(1-a/b)}
\end{split}
\end{equation}
Set $S_1 =  \sum_{j\geq1} \|L_j * f\|_2^{pa/b} \cdot I_1^{p/2(1-a/b)} $ and $S_2 = \sum_{j\geq1} \|L_j * f\|_2^{pa/b} \cdot I_2^{p/2(1-a/b)}$. Then it is enough to show that $S_1 \lesssim1 $ and $S_2 \lesssim 1$.
First we estimate $I_1$  with $L^2$estimates \eqref{L2} as follows
\begin{eqnarray*}
I_1 & \lesssim &  \int_{\mathbb{H}^n} | f * L_j (x)|^2 dx\cdot R^{2(2n+2)b } \lesssim  2^{2j(\alpha - n\beta)} \| f\|_2^{2} \cdot R^{2(2n+2)b }
\\
& \leq &  2^{2j(\alpha- n\beta)}  R^{2(2n+2)b } \cdot R^{(2n+2) (1 - 2/p)}=  2^{2j(\alpha - (n+1/2) \beta)} R^{2(2n+2) \delta}.
\end{eqnarray*}
Thus we can bound $\|L_j * f\|_2^{a/b} \cdot I_1^{\frac{1}{2}(1-a/b)}$ as
\begin{eqnarray*}
\| L_j * f \|_2^{a/b} \cdot I_1^{1/2(1-a/b)} & \lesssim& \left\{ 2^{j(\alpha-n \beta)} R^{(2n+2)(1/2-1/p)} \right\}^{a/b} \cdot \left\{ 2^{j(\alpha - n \beta)} \cdot R^{(2n+2) \delta} \right\}^{(1-a/b)}
\\
& =& 2^{j (\alpha-n \beta)},
\end{eqnarray*}
and we have $S_1 \lesssim \sum_{ j \geq 1} 2^{j(\alpha-n\beta)p} \lesssim 1.$

For $I_2$ we consider the two cases $R >1$ and $R\leq 1$.

\noindent $Case~ (i)$:  Suppose $R>1$. In the integral
\begin{eqnarray*}
(L_j * f) (x) = \int L_j (x y^{-1}) f (y) dy,
\end{eqnarray*}
we have $|xy^{-1}| \leq 2^{j} $ and $|y| \leq R$, which imply $|x| \leq |xy^{-1}| + |y| \leq 2^{j} + R$. Therefore, in \eqref{28}, we have that  $I_2 =0$ for $2^{j} <R$. Thus we only need to consider $j$ with $2^{j} \geq R$. Then we have $|x| \leq 2^{j+1}$ for $x$ in the support of $L_j * f $, and so
\begin{eqnarray}\label{I_2}
I_2 \lesssim \int_{|x|>2R} | f * L_j (x)|^2  dx \cdot 2^{2(2n+2)bj}.
\end{eqnarray}
By \eqref{formula} we have 
\begin{eqnarray*}
| L_j (xy^{-1}) - P_j^{x} (y) | & \lesssim & |y|^{I+1} \sup_{|\alpha| \leq I+1} | X^{\alpha} L_j (xy^{-1})|
\\
& \lesssim & |y|^{I+1} 2^{-j(2n+2-\alpha)} 2^{j(\beta-1)(I+1)}.
\end{eqnarray*}
Since $f(y)$ has support in $|y| \leq R$ and \eqref{new}, we have
\begin{eqnarray*}
|L_j * f (x)| &\lesssim &R^{I+1} 2^{-j( 2n+2-\alpha)} 2^{j(\beta-1)(I+1)} \int_{|y| \leq R} |f(y)| dy
\\
&\lesssim &R^{I+1} 2^{-j(2n+2-\alpha)} 2^{-j(\beta-1)(I+1)} R^{\frac{1}{2}(2n+2)} \|f\|_2.
\\
&\lesssim &  2^{-j(2n+2-\alpha)} (R 2^{-j(\beta-1)})^{(I+1)} R^{(2n+2)(1-\frac{1}{p})}.
\end{eqnarray*}
Thus we can estimate \eqref{I_2} as
\begin{eqnarray*}
I_2 &\lesssim& 2^{2(2n+2)bj} 2^{j(2n+2)} \left\{2^{-j(2n+2-\alpha)} (R 2^{-j(\beta-1)})^{(I+1)} R^{(2n+2)(1-\frac{1}{p})} \right\}^2.
\\
&=& 2^{2j \{(2n+2)(1/p + 1 + \delta) + \alpha\}} (R 2^{j(\beta-1)})^{2(I+1)} R^{2(2n+2)(1-\frac{1}{p})}.
\end{eqnarray*}
Here we may choose $I=0$ and $I=s$, which gives
\begin{eqnarray*}
I_2 \lesssim 2^{2j \{(2n+2)(1/p -1+\delta) + \alpha\}} R^{2(2n+2)(1-\frac{1}{p})} \min \{ 1, (R 2^{j(\beta-1)})^{2(s+1)} \}.
\end{eqnarray*}
Thus,
\begin{equation}\label{15}
\begin{split}
\|L_j*f\|_2^{a/b} \cdot I_2^{\frac{1}{2}(1-a/b)} \lesssim \{ 2^{j(\alpha -n\beta)} R^{(2n+2)(1/2-1/p)}\}^{a/b} &
\\
 \cdot \{2^{j \{(2n+2)(1/p -1+\delta) + \alpha\}} R^{(2n+2)(1-\frac{1}{p})} & \min\left( 1, (R 2^{j(\beta-1)})^{(s+1)}\right) \}^{(1-a/b)}.
\end{split}
\end{equation}
Provided $\delta$ is small enough, we have
\begin{eqnarray*}
(2n+2)(\frac{1}{p} - 1 + \delta) + \alpha + (\beta -1)(s+1) &=&(2n+2)(\frac{1}{p} -1 + \delta) + \alpha + (\beta-1) ([(2n+2)(\frac{1}{p}-1)]+1)
\\
& < & (2n+2)(\frac{1}{p} -1 +\delta) + \alpha + (\beta-1)(2n+2)(\frac{1}{p} -1)
\\
&=& (2n+2)(\frac{1}{p} -1)\beta + \alpha + (2n+2) \delta< 0.
\end{eqnarray*}
Therefore the index of $2^{j}$ in \eqref{15} with $(R 2^{j(\beta-1)})^{s+1}$ is negative for small $\delta >0$. Remind that $R>1$. Then, from (1) in Lemma \ref{sum} we have
\begin{eqnarray*}
\sum_{j \geq 1} \|L_j * f\|_2^{pa/b} \cdot I_2^{\frac{p}{2}(1-a/b)} &\lesssim& R^{p \mu_{\delta}} +  \log(R+1) R^{p \kappa_{\delta}},
\end{eqnarray*}
where
\begin{eqnarray*}
R^{p \mu_{\delta}} &=&  R^{(2n+2)(1/2-1/p) \frac{pa}{b} + (2n+2)(1-1/p) p(1-a/b)} ,
\\
R^{p \kappa_{\delta}} &=& [R^{-\frac{1}{1-\beta}[\alpha - n \beta]} R^{(2n+2)(1/2-1/p)}]^{p\delta/b} \cdot [R^{\frac{1}{1-\beta}\{(2n+2)(1/p-1+\delta) + 2\alpha\}} \cdot R^{(2n+2)(1-1/p)}]^{p(1-a/b)}.
\end{eqnarray*}
Because $p \leq 1$, we easily see that $\mu_{\delta} \leq 0$. Moreover,
\begin{eqnarray*}
\kappa_{0} = \frac{1}{1-\beta}\{ \beta(2n+2)(\frac{1}{p} -1) + \alpha \} < 0.
\end{eqnarray*}
From this, we get $\kappa_{\delta} < 0$ for $\delta$ small enough. Therefore we have
\[
S_2 \lesssim R^{\mu_{\delta}} + \log (R+1) R^{\kappa_{\delta}} \lesssim 1.\]

\noindent $Case~ (ii)$: Suppose $R \leq 1.$ We see that $\textrm{min} ( 1, (R 2^{j(\beta-1)(s+1)})) = R 2^{j (\beta-1)(s+1)}$and \eqref{15} becomes  
\begin{eqnarray*}
\sum_{j\geq1} \|L_j*f\|_2^{pa/b} \cdot I_2^{\frac{p}{2}(1-a/b)} \lesssim \{ 2^{j(\alpha-n\beta)} R^{(2n+2)(1/2-1/p)}\}^{pa/b} 
\\
\{2^{j(2n+2)(1/p-1+ \delta) + \alpha} R^{(2n+2)(1-\frac{1}{p})} \cdot (R2^{j(\beta-1)})^{(s+1)}\}^{p(1-a/b)}.
\end{eqnarray*}
Because the power of $2^j$ is negative, provided $\delta$ is small enough, we get 
\begin{eqnarray*}
\sum_{j\geq1} \|L_j * f\|_2^{pa/b}\cdot I_2^{\frac{p}{2}(1-a/b)} &\lesssim& R^{(2n+2)(1/2-1/p) \frac{pa}{b}} \cdot R^{\{ (2n+2)(1-\frac{1}{p}) + (s+1)\} p(1-\frac{a}{b})}
\\
&=:& R^{p \mu_{\delta}}.
\end{eqnarray*}
Observe that
\begin{eqnarray*}
{\mu_{0}} = (2n+2)(1-\frac{1}{p} ) + (s+1) = (2n+2)(1-\frac{1}{p}) + ([(2n+2)(\frac{1}{p}-1)] +1)> 0.
\end{eqnarray*}
Thus we have $\mu_{\delta} > 0$ for $\delta$ small enough. Now we get
\[
\sum_{j \geq 1} \|L_j * f\|_2^{pa/b} \cdot I_2^{\frac{p}{2}(1-a/b)} \lesssim R^{p \mu_{\delta}} \leq 1.
\]
We then conclude that $S_2 \lesssim 1$. The proof is complete.
\end{proof}
We now establish the same result for the case $\beta > 1$.
\begin{thm}
For $1 < \beta $, $p \leq 1$, if $(\frac{1}{p} -1)(2n+2) \beta + \alpha < 0$,
the operator $T_{L_{\alpha,\beta}}$ is bounded on $H^p$ space. 
\end{thm}

\begin{proof}
By arguing as in  \eqref{Lsplit1}--\eqref{Lsplit2} in the proof of Theorem \ref{L1} to obtain the following
\begin{equation}\label{Lsplit?}
\begin{split}
\sum_{j\geq1} \|L_j * f\|_{H^p}^p \lesssim& \sum_{j\geq1} \|L_j * f\|_2^{pa/b} \cdot I_1^{p/2(1-a/b)} + \sum_{j\geq1} \|L_j * f\|_2^{pa/b} \cdot I_2^{p/2(1-a/b)},
\end{split}
\end{equation}
 where $I_1$ and $I_2$ are defined as in \eqref{28}.
 Because the estimate for $I_1$  is exactly same with the proof of Theorem \ref{L1}, we only deal with $I_2$. As before, we have

\begin{equation}\label{K4}
\begin{split}
\|L_j*f\|_2^{a/b} \cdot I_2^{\frac{1}{2}(1-a/b)} \lesssim \{ 2^{j(\alpha -n\beta)} R^{(2n+2)(1/2-1/p)}\}^{a/b} &
\\
 \cdot \{2^{j \{(2n+2)(1/p -1+\delta) + \alpha\}} R^{(2n+2)(1-\frac{1}{p})} & \min\left( 1, (R 2^{j(\beta-1)})^{(s+1)}\right) \}^{(1-a/b)}
\end{split}
\end{equation}

\noindent$Case~ (i)$: Suppose $R>1$.
As for the case $\beta < 1$, we have $I_2 =0$ if $2^j < R$ and we only need  consider $j$ with $2^j \geq R$. Since $R 2^{j (\beta-1)} \geq 1$, we estimate $I_2$ as
\begin{eqnarray*}
I_2 \lesssim 2^{j\{2(2n+2)[1/p-1+ \delta] + 2\alpha\}}R^{2(2n+2)(1-1/p)}.
\end{eqnarray*}
Note that
\begin{eqnarray}\label{2n+2}
(2n+2)(\frac{1}{p} -1) + \alpha < (2n+2) (\frac{1}{p} -1 ) \beta + \alpha < 0. 
\end{eqnarray}
Thus, if $\delta$ is sufficiently small,  we have $(2n+2)(1/p - 1 + \delta) + \alpha < 0$ and we can sum \eqref{K4} as 
\begin{eqnarray}\label{R}
\sum_{j\geq1}\|L_j * f\|_2^{pa/b}\cdot I_2^{\frac{p}{2}(1-a/b)} \lesssim R^{(2n+2)(1/2-1/p)\frac{pa}{b}} \cdot R^{\{(2n+2)(1-\frac{1}{p})\}p(1-\frac{a}{b})} \leq 1,
\end{eqnarray}
where the last inequality holds because $ p\leq 1 $ and $R>1$.

\noindent$Case ~(ii)$: Suppose $R \leq1$. From \eqref{2n+2}, using (1) in Lemma \ref{sum} we have
\begin{eqnarray*}
\sum_{j\geq1} \|L_j * f\|_2^{pa/b}\cdot I_2^{\frac{1}{2}p(1-a/b)} &\lesssim& R^{p \mu_{\delta}} + |\log R| ~R^{p \kappa_{\delta}},
\end{eqnarray*}
where
\begin{eqnarray*}
R^{\mu_{\delta}} &=& R^{(2n+2)(1/2-1/p)\frac{a}{b}} \cdot R^{\{(2n+2)(1-1/p) + (s+1)\}(1-\frac{a}{b})},
\\
R^{\kappa_{\delta}} &=&  R^{(2n+2)(1/2-1/p)\frac{a}{b}} \cdot \{R^{\frac{1}{1-\beta} \{ (2n+2) (1/p -1 +\delta) + \alpha\}} R^{(2n+2) (1-1/p)}\}^{1-a/b}.
\end{eqnarray*}
Observe that
\begin{eqnarray*}
\mu_0 = (2n+2)(1-\frac{1}{p}) + (s+1) = (2n+2)(1-\frac{1}{p}) + [(2n+2)(\frac{1}{p}-1)]+1 >0
\end{eqnarray*}
and
\begin{eqnarray*}
\kappa_0 = \frac{1}{1-\beta} \{ (2n+2) (1/p -1 ) + \alpha\} + (2n+2) (1- 1/p) = \frac{1}{1-\beta} \{ \beta (2n+2)(\frac{1}{p} -1) + \alpha\} > 0.
\end{eqnarray*}
Therefore we have $\mu_{\delta}, \kappa_{\delta} > 0$ for $\delta$ small enough, and so 
\begin{eqnarray}\label{R2}
\sum_{j\geq 1} \|L_j * f\|_2^{a/b} \cdot I_2^{\frac{1}{2}(1-a/b)} \lesssim R^{\mu_{\delta}} + |\log R| \cdot R^{\kappa_{\delta}} \leq 1.
\end{eqnarray}
Now we conclude that $S_2 \lesssim 1$ from \eqref{R} and \eqref{R2}. The proof is complete.
\end{proof}

\section{Necessary conditions}
In this section we show that the Hardy space boundedness obtained in the previous section is sharp except for the endpoint cases. We only give an example for Theorem \ref{K}. Examples for the other theorems can be found similarly. We refer to Sj\'olin \cite{Sj3} for the Euclidean case.
\\
We let  $g(x)$ a function such that 
\begin{eqnarray*}
\int_{\mathbb{R}} x^{\alpha} g(x) dx = 0 \quad \textrm{for}~ 0\leq \alpha \leq k \quad \textrm{and} \quad \int_{\mathbb{R}} x^{k+1} g(x) dx \neq 0.
\end{eqnarray*}
Let  $h(x_2,\dots,x_{2n},x_{2n+1})$ a function supported on the ball $B(0,1)$ such that $\int_{\mathbb{R}^{2n}} h \neq 0$ and let $f$ be the function on $\mathbb{R}^{2n+1}$ defined by  $f(x_1,\dots,x_{2n+1}) = g(x_1) h(x_2,\dots,x_{2n+1}) \forall (x_1,\cdots,x_{2n+1})\in \mathbb{R}^{2n+1}$. Then 
\begin{eqnarray*}
\int_{\mathbb{H}^{n}} x^{\alpha} f(x) = 0 , \quad \textrm{if}\quad |\alpha|\leq k.
\end{eqnarray*} 
For $\epsilon >0$ set $f_{\epsilon}(x) = \epsilon^{-(2n+2)/p} f (\frac{x}{\epsilon})$. We note that $\| f_{\epsilon}\|_{H^p} = C~ \textrm{for all } \epsilon>0$. Assume that $T_{K_{\alpha,\beta}}$ is bounded on $H^p$. Then $\| T_{K_{\alpha,\beta}} (f_{\epsilon})\|_{H^p} \lesssim 1$. Note that $|y| \leq \epsilon$ for $y \in \textrm{supp}(f_{\epsilon})$. Then, for $|x| \geq C \epsilon$ with a large constant $C >0$, we have
\begin{eqnarray*}
K * f (x) &=& \int K(x y^{-1}) f_{\epsilon}(y) dy
\\
&=& \int \left( K(x y^{-1}) - \sum_{|\alpha|\leq k+1} \frac{1}{\alpha !} D^{\alpha} K(x) y^{\alpha}\right) f_{\epsilon} (y) dy + \int \left(\sum_{|\alpha|\leq k+1} \frac{1}{\alpha !} D^{\alpha} K(x) y^{\alpha} \right) f_{\epsilon} (y) dy
\\
&=& \int D^{k+2} K(x y_*^{-1}) O (y^{k+2}) f_{\epsilon}(y) dy + C \partial_{x_1}^{k+1}K(x) \int_{\mathbb{R}} y_1^{k+1} f_{\epsilon} (y_1) dy_1, \quad \quad |y_{*}|\leq |y|\leq \epsilon
\\
&=& O(\epsilon^{(2n+2)+k+2 -\frac{(2n+2)}{p}} |x|^{-(n+\alpha + (k+2)(\beta+1))}) + \epsilon^{k+1+(2n+2)-\frac{(2n+2)}{p}} \partial_{x_1}^{k+1} K(x).
\end{eqnarray*}
Take $K (x) = |x|^{-2n-2-\alpha} e^{i|x|^{-\beta}}\chi(x).$ We see that $|\partial_{x_1}^{k+1} K(x)| \sim |x|^{-(2n+2)-\alpha - (k+1){(\beta+1)}} $ for small $x$. 
For $ \epsilon \lesssim |x|^{\beta+1}$ we have
\begin{eqnarray*}
\epsilon^{(2n+2) + k+2 - \frac{2n+2}{p} } |x|^{-(2n+2+\alpha +(k+2)(\beta+1))} \lesssim \epsilon^{(2n+2)+(k+1) - \frac{(2n+2)}{p} }|x|^{-(2n+2)-\alpha -(k+1)(\beta+1)}.
\end{eqnarray*}
Therefore we get
\begin{eqnarray*}
K_{\alpha, \beta} * f_{\epsilon} (x) \sim \epsilon^{(2n+2)+k+1-\frac{(2n+2)}{p}}|x|^{-(2n+2)-\alpha - (k+1)(\beta+1)} \quad \textrm{for}~|x| \gtrsim \epsilon^{1/{(\beta+1)}}.
\end{eqnarray*}
Then,
\begin{eqnarray*}
1 \gtrsim \int_{\mathbb{H}^n} |K_{\alpha,\beta} * f_{\epsilon} (x)|^p dx &\gtrsim& \epsilon^{p(2n+2)+kp+p-(2n+2)} \int_{c\geq |x| \gtrsim \epsilon^{1/(\beta+1)}}  |x|^{-(2n+2)p-\alpha p - (k+1)(\beta+1)p} dx 
\\
&\gtrsim& \epsilon^{p(2n+2)+kp+p-(2n+2)} \epsilon^{-\frac{(2n+2)p-(2n+2)+\alpha p}{\beta+1} - (k+1) p} 
\\
&=& \epsilon^{\frac{-p}{\beta+1} \left[(\frac{1}{p}-1)(2n+2) \beta + \alpha  \right]}.
\end{eqnarray*}
This implies that $ (1-\frac{1}{p})(2n+2) \beta + \alpha$ must be $\leq 0$. This shows that Theorem \ref{K} is sharp except the endpoint case $ (1-\frac{1}{p})(2n+2) \beta + \alpha =0$.

\section*{Acknowledgements}\thispagestyle{empty} 
I am deeply grateful to my advisor Rapha\"el Ponge for his support and careful proofreading during the preparation of this paper.

\end{document}